\documentclass[11pt]{article}
\date{}
\title{Graph classes with linear Ramsey numbers\thanks{Some results presented in this paper appeared in the extended abstract 
\cite{IWOCA2018} published in the proceedings of the 29th International Workshop on Combinatorial Algorithms, IWOCA 2018.}} 
\author{Bogdan Alecu\thanks{Mathematics Institute, University of Warwick, Coventry, CV4 7AL, UK. B.Alecu@warwick.ac.uk} \and
Aistis Atminas\thanks{Department of Mathematical Sciences, Xi'an Jiaotong-Liverpool University, 111 Ren'ai Road, Suzhou 215123, China. Email: Aistis.Atminas@xjtlu.edu.cn} \and 
Vadim Lozin\thanks{Mathematics Institute, University of Warwick, Coventry, CV4 7AL, UK. V.Lozin@warwick.ac.uk} \and 
Viktor Zamaraev\thanks{Department of Computer Science, University of Liverpool, Ashton Building, Ashton Street, Liverpool, 
L69 3BX, UK. Email: Viktor.Zamaraev@liverpool.ac.uk}
}
\usepackage{graphicx}
\usepackage{tabularx}
\usepackage{url}
\usepackage{amsthm} 
\usepackage{amsmath} 
\usepackage{amsfonts}
\usepackage{amssymb}
\usepackage{tikz}
\usepackage{hhline}
\usepackage{array}
\usetikzlibrary{decorations.pathreplacing}
\tikzstyle{vertex}=[circle,fill=black!100,text=white,inner sep=0.8mm]
\tikzstyle{point}=[circle,fill=black,inner sep=0.1mm]

\begin{document}
\maketitle

\newtheorem{proposition}{Proposition}
\newtheorem*{example}{Example}
\newtheorem{theorem}{Theorem}
\newtheorem{lemma}{Lemma}
\newtheorem{cor}{Corollary}
\newtheorem{claim}{Claim}
\theoremstyle{definition}
\newtheorem{definition}{Definition}
\newtheorem{remark}{Remark}
\newtheorem{conjecture}{Conjecture}
\newtheorem{question}{Question}
\newtheorem{notation}{Notation}
\newtheorem{statement}{Statement}
\newtheorem{observation}{Observation}

\begin{abstract}
The Ramsey number $R_X(p,q)$ for a class of graphs $X$ is the minimum $n$ such that
every graph in $X$ with at least $n$ vertices has either a clique of size $p$ or
an independent set of size $q$. We say that
Ramsey numbers are {\it linear} in $X$ if there is a constant $k$ such that $R_{X}(p,q) \leq k(p+q)$ for all $p,q$.
In the present paper we conjecture that if $X$ is a hereditary class defined by finitely many forbidden induced subgraphs, 
then Ramsey numbers are linear in $X$ if and only if $X$ excludes a forest, a disjoint union of cliques and their complements. We prove the ``only if'' part of this conjecture and verify the ``if'' part for  a variety of 
classes. We also apply the notion of linearity to bipartite Ramsey numbers and reveal a number of similarities and differences 
between the bipartite and non-bipartite case. 
\end{abstract}

\section{Introduction}

According to Ramsey's Theorem \cite{Ramsey}, for all natural $p$ and $q$
there exists a minimum number $R(p,q)$ such that every graph with at least $R(p,q)$ vertices has either a clique of size $p$ or an independent set of size $q$.

The exact values of Ramsey numbers are known only for small values of $p$ and $q$. 
However, with the restriction to 
specific classes of graphs, Ramsey numbers can be determined for all $p$ and $q$. 
In particular, in \cite{planar} this problem was solved for planar graphs, while in \cite{first}
it was solved for line graphs, bipartite graphs, perfect graphs, $P_4$-free graphs and 
some other classes. 

These studies reveal, in particular, that different classes have different rates of growth of Ramsey numbers.
In the present paper, we denote the Ramsey numbers restricted to a class $X$ by $R_X(p,q)$ and focus 
on classes with a smallest speed of growth of $R_X(p,q)$. Clearly, $R_X(p,q)$ cannot be smaller than 
the minimum of $p$ and $q$. We say that Ramsey numbers are {\it linear} in $X$ if there is a constant $k$ 
such that $R_{X}(p,q) \leq k(p+q)$ for all $p,q$.

All classes in this paper are {\it hereditary}, i.e., closed under taking induced subgraphs.
It is well known that a class of graphs is hereditary if and only if it can be characterized 
in terms of minimal forbidden induced subgraphs. If the number of  minimal forbidden induced subgraphs
for a class $X$ is finite, we say that $X$ is {\it finitely defined}.

It is not difficult to see that all classes of bounded co-chromatic number have linear Ramsey numbers, where 
the co-chromatic number of a graph $G$ is the minimum $k$ such that the vertex set of $G$ can be partitioned 
into $k$ subsets each of which is either a clique or an independent set. Unfortunately, as we show in Section \ref{sec:ccn}, 
this is not an if and only if statement in general. We conjecture, however, that in the universe of finitely defined classes the two notions coincide.

\begin{conjecture}\label{con:1}
A finitely defined hereditary class is of linear Ramsey numbers if and only if it is of bounded co-chromatic number.  
\end{conjecture}

In \cite{cs}, it was conjectured that a finitely defined class $X$ has bounded co-chromatic number if and only if 
the set of minimal forbidden induced subgraphs for $X$ contains a $P_3$-free graph, the complement of a $P_3$-free graph,
a forest (i.e., a graph without cycles) and the complement of a forest. The authors of \cite{cs} go on to show that their conjecture is equivalent to the older Gy\'{a}rf\'{a}s-Sumner conjecture \cite{chi, sumner}. Naturally, if this conjecture is true, we expect, following Conjecture~\ref{con:1}: 

\begin{conjecture}\label{con:2}
	A finitely defined  class $X$ is of linear Ramsey numbers if and only if the set of minimal forbidden induced subgraphs for $X$ contains a $P_3$-free graph, the complement of a $P_3$-free graph,
	a forest and the complement of a forest.  
\end{conjecture} 

In Section~\ref{sec:non}, we prove the ``only if'' part of Conjecture~\ref{con:2}.  In other words, in the universe of finitely defined classes, 
the property of a class $X$ having linear Ramsey numbers lies in between that of $X$ having bounded co-chromatic number and that of $X$ avoiding the specified induced subgraphs. 

In Section~\ref{sec:linear}, we focus on the ``if'' part of  Conjecture~\ref{con:2} and verify it for a variety of classes defined by small forbidden induced subgraphs. 
Moreover, for all the considered classes we derive exact values of the Ramsey numbers.

In Section~\ref{sec:bip}, we extend the notion of linearity to bipartite Ramsey numbers and show that some of the results obtained for non-bipartite numbers
can be extended to the bipartite case as well. However, in general, the situation with linear bipartite Ramsey numbers seems to be more complicated 
and we restrict ourselves to a weaker analog of Conjecture~\ref{con:2}, which is also verified for some classes of bipartite graphs. 
In the rest of the present section, we introduce basic terminology and notation.

%\medskip
All graphs in this paper are finite, undirected, without loops and multiple edges. The vertex set and the edge set of a graph $G$
are denoted by $V(G)$ and $E(G)$, respectively. For a vertex $x\in V(G)$ we denote by $N(x)$ the neighbourhood of $x$, i.e., the set of 
vertices of $G$ adjacent to $x$. The degree of $x$ is $|N(x)|$.
We say that $x$ is {\it complete} to a subset $U\subset V(G)$ if $U\subseteq N(x)$ and {\it anticomplete} to $U$ if $U\cap N(x)=\emptyset$.
A subgraph of $G$ induced by a subset of vertices $U\subseteq V(G)$ is denoted $G[U]$.
By $\overline{G}$ we denote the complement of $G$ and call it co-$G$.

A {\it clique} in a graph is a subset of pairwise adjacent vertices and an {\it independent set} is a subset of pairwise non-adjacent vertices. 
For a graph $G$, let $\alpha(G)$ denote the independence number of $G$, $\omega(G)$ the clique number, $\chi(G)$ the chromatic number and $z(G)$ the co-chromatic number.

By $K_n$, $C_n$ and $P_n$ we denote a complete graph, a chordless cycle and a chordless path with $n$ vertices, respectively.
Also, $K_{n,m}$ is a complete bipartite graph with parts of size $n$ and $m$, and $K_{1,n}$ is a star. 
A disjoint union of two graphs $G$ and $H$ is denoted $G+H$. In particular, $pG$ is a disjoint union of $p$ copies of $G$.

If a graph $G$ does not contain induced subgraphs isomorphic to a graph $H$, then we say that $G$ is $H$-free and call $H$ a forbidden induced subgraph for $G$.
In case of several forbidden induced subgraphs we list them in parentheses. 

A {\it bipartite graph} is a graph whose vertices can be partitioned into two independent sets, and a {\it split graph} is a graph whose vertices can be partitioned 
into an independent set and a clique. A graph is bipartite if and only if it is free of odd cycles, and a graph is a split graph if and only if it is $(C_4,2K_2,C_5)$-free \cite{split}.

\section{Linear Ramsey numbers and related notions}
\label{sec:ccn}

As we observed in the introduction, the notion of linear Ramsey numbers has ties with bounded co-chromatic number,
and we believe that in the universe of finitely defined classes, the two notions are equivalent. In the present section, 
we first show that this equivalence is not valid for general hereditary classes, and then discuss the relationship between 
linear Ramsey numbers and some other notions that appear in the literature.

In order to show that Conjecture~\ref{con:1} is not valid for general hereditary classes, we consider the Kneser graph $KG_{a,b}$:
it has as vertices the $b$-subsets of a set of size $a$, and two vertices are adjacent if and only if the corresponding subsets are disjoint. 
A well-known result due to Lov\'asz says that, if $a \geq 2b$, then the chromatic number $\chi(KG_{a, b})$ is $a - 2b + 2$ \cite{lovasz1978}. 

In the following theorem, we denote by $X$ the hereditary closure of the family of Kneser graphs $KG_{3n, n}, n \in \mathbb{N}$,  i.e., $X = \{H: H \text{ is an
induced subgraph of } KG_{3n, n},$ $\text{for some } n \in \mathbb N\}$. 
 
\begin{theorem}
The class $X$ has linear Ramsey numbers and unbounded co-chromatic number.
\end{theorem}

\begin{proof}
First, we note that by Lov\'asz's result stated above, it follows that $\chi(KG_{3n, n}) = 3n-2n+2=n+2$. Also, it is not hard to see that 
the the size of the biggest clique in $KG_{3n, n}$ is 3. It follows that the co-chromatic number of $KG_{3n, n}$ is at least $\frac{n+2}{3}$. As a result, 
the co-chromatic number is unbounded for this class.

Now consider any induced subgraph $H$ of $KG_{3n,n}$. We will show that $\alpha(H) \geq \frac{|V(H)|}{3}$. Indeed, the vertices of the Kneser graph in this case
are $n$-element subsets of $\{1, 2, \ldots, 3n\}$. For each $i \in \{1,2, \ldots, 3n\}$ let $V_i$ be the set of vertices of $H$ containing element $i$. 
Then, as each vertex is an $n$-element subset, it follows that $\sum_{n=1}^{3n} |V_i| = n \times |V(H)|$. Hence, by the Pigeonhole Principle, there is an $i$ such that 
$|V_i| \geq \frac{|V(H)|}{3}$. As $V_i$ is an independent set, it follows that $\alpha(H) \geq \frac{|V(H)|}{3}$. This implies that for any $H \in X$  we have $|V(H)| \leq 3 \alpha(H) \leq 3(\alpha(H)+\omega(H))$,
and hence the Ramsey numbers are linear in the class $X$.
\end{proof}

\medskip
We now turn to one more notion, which is closely related to the growth of Ramsey numbers. This is the notion of homogeneous subgraphs that appears 
in the study of the Erd\H{o}s-Hajnal conjecture \cite{EH-conjecture}. 
We will say that graphs in a class $X$ have \emph{linear homogeneous subgraphs} if there exists a constant $c = c(X)$ such that
$\max\{ \alpha(G), \omega(G) \} \geq c \cdot |V(G)|$ for every $G \in X$.

\begin{proposition}\label{LLRN}
	Let $X$ be a class of graphs. Then graphs in $X$  have linear homogeneous subgraphs if and only if Ramsey numbers are linear in $X$. More generally, for any $0 < \delta \leq 1$, the following two statements are equivalent:
	\begin{itemize}
		
		\item There is a constant $A$ such that $\max\{\alpha(G), \omega(G)\} \geq A \cdot |V(G)|^\delta$ for every $G \in X$.
		
		\item There is a constant $B$ such that $R_X(p, q) \leq B(p + q)^{\frac{1}{\delta}}$. 

	\end{itemize}
\end{proposition}

\begin{proof}
	
The second claim reduces to the first one when $\delta = 1$, so we just prove the stronger claim. 
	
For the first implication, suppose there exists a constant $A$ such that $\max\{ \alpha(G), \omega(G) \} \geq A \cdot |V(G)|^\delta$ for all $G \in X$. 
Let $H \in X$, let $p, q \in \mathbb{N}$, and suppose that $|V(H)| \geq \left(\frac{p + q}{A}\right)^{\frac{1}{\delta}}$. 
Then $\max\{ \alpha(H), \omega(H) \} \geq A \cdot |V(H)|^\delta \geq p+q$, which means that $H$ is guaranteed to have 
an independent set of size $p$ or a clique of size $q$, and this proves the first direction (we can put, e.g., $B = A^{-\frac{1}{\delta}}$) in the statement of the proposition.

Conversely, suppose there exists a positive constant $B$ such that for any $p, q \in \mathbb{N}$ and 
$G \in X$, if $|V(G)| \geq B(p + q)^{\frac{1}{\delta}}$, 
then $G$ has an independent set of size $p$ or a clique of size $q$. 
Let $H$ be an  arbitrary graph in $X$ and let $t$ be the largest integer such that $|V(H)| \geq 2^{\frac{1}{\delta}}Bt^{\frac{1}{\delta}} = B(t + t)^{\frac{1}{\delta}}$.
By the above assumption, $H$ has a clique or an independent set of size $t$, i.e., $\max\{\alpha(H), \omega(H)\} \geq t$.
Notice, by definition of $t$, we have $|V(H)| \leq 2^{\frac{1}{\delta}}B(t + 1)^{\frac{1}{\delta}}$, i.e., $|V(H)|^{\delta} \leq 2B^\delta(t + 1)$. Hence if $t=0$, then $|V(H)|^\delta \leq 2B^\delta$  and therefore 
$\max\{ \alpha(H), \omega(H) \} \geq \frac{|V(H)|^\delta}{2B^\delta} \geq \frac{|V(H)|^\delta}{4B^\delta}$.
On the other hand, if $t \geq 1$, then $|V(H)|^\delta \leq 2B^\delta(t+1) \leq 4B^\delta t$ and therefore 
$\max\{ \alpha(H), \omega(H) \} \geq \frac{|V(H)|^\delta}{4B^\delta}$, and putting, e.g., $A = \frac{1}{4B^\delta}$ concludes the proof. 	
\end{proof}

In particular, the Erd\H{o}s-Hajnal conjecture can be stated in our terminology as follows: 

\begin{conjecture}(Erd\H{o}s-Hajnal)
	Suppose $X$ is a proper hereditary class (that is, not the class of all graphs). Then there are constants $A, k$ such that $R_X(p, q) \leq A(p + q)^k$ for every $p, q \in \mathbb{N}$, i.e., Ramsey numbers grow at most polynomially in $X$. 
	
\end{conjecture}

\medskip
Finally, we point out the difference between the notion of Ramsey numbers {\it for  classes} and  the notion of Ramsey numbers {\it of graphs}. 
Each of them leads naturally to the notion of linear Ramsey numbers, defined differently in the present paper and, for instance, in \cite{GRR00}. 
In spite of the possible confusion, we use the terminology of Ramsey numbers, and not the terminology of homogeneous subgraphs, 
because most of our results deal with the exact value of $R_X(p,q)$.

\section{Classes with non-linear Ramsey numbers}
\label{sec:non}

%%%%%%%%%%%%%%%%%%%%%%%%%%%%%%%%%%%%%%%%%%%%%%%%%%%%%%%%%%%%%%%%%%%%%%%%%%%%%%%%%
%%%%%%%%%%%%%%%%%%%%%%%%%%%%%%%%%%%%%%%%%%%%%%%%%%%%%%%%%%%%%%%%%%%%%%%%%%%%%%%%%

In this section, we prove the ``only if'' part of Conjecture~\ref{con:2}. 

\begin{lemma}\label{lem:E}
	For every fixed $k \geq 3$, the class $X_k$ of $(C_3,C_4,\ldots,C_k)$-free graphs is not of linear Ramsey numbers.
\end{lemma}
\begin{proof}
	Assume to the contrary that Ramsey numbers for the class $X_k$ are linear.  
	Then, since graphs in $X_k$ do not contain cliques of size three, there exists a constant $t=t(k)$ 
	such that any $n$-vertex graph from the class has an independent set of size at least $n/t$. 
	
	It is well-known (see, e.g., \cite{method}) that  $X_k$ contains $n$-vertex graphs with the independence number of order $O(n^{1-\epsilon} \ln n)$, 
	where $\epsilon>0$ depends on $k$, which is smaller than $n/t$ for large $n$. 
	This contradiction shows that $X_k$ is not of linear Ramsey numbers.
\end{proof}

\begin{theorem}\label{thm:onlyif}
	Let $X$ be a class of graphs defined by a finite set $M$ of forbidden induced subgraphs.
	If $M$ does not contain a graph in at least one of the following four classes, then $X$ is not of linear Ramsey numbers: 
	$P_3$-free graphs, the complements of $P_3$-free graphs, forests, the complements of forests. 
\end{theorem}

\begin{proof}
	It is not difficult to see that a graph is $P_3$-free if and only if it is a disjoint union of cliques. 
	The class of $P_3$-free graphs contains  the graph $(q-1)K_{p-1}$ with $(q-1)(p-1)$ vertices and with no clique of size $p$ or independent set of size $q$,
	and hence this class is not of linear Ramsey numbers. Therefore, if $M$ contains no $P_3$-free graph, then $X$ contains all $P_3$-free graphs and hence is not of linear Ramsey numbers.
	Similarly, if $M$ contains no $\overline{P}_3$-free graph, then $X$ is not of linear Ramsey numbers.
	
	Now assume that $M$ contains no forest. Therefore, every graph in $M$ contains a cycle. Since the number of graphs in $M$ is finite, 
	$X$ contains the class of $(C_3,C_4,\ldots,C_k)$-free graphs for a finite value of $k$ and hence is not of linear Ramsey numbers by Lemma~\ref{lem:E}.
	Applying the same arguments to the complements of graphs in $X$, we conclude that if $M$ contains no co-forest, then $X$ is not of linear Ramsey numbers. 
\end{proof}

\section{Classes with linear Ramsey numbers}
\label{sec:linear}

In this section, we study classes of graphs defined by forbidden induced subgraphs with 4 vertices and determine Ramsey numbers for several classes in this family 
that  verify the ``if'' part of Conjecture~\ref{con:2}.
All the eleven graphs on 4 vertices are represented in Figure~\ref{fig:four-vertex}.
% \begin{figure}[ht]
% \centering
% \includegraphics[width=80mm]{four-vertex.jpg}
% \caption{All 4-vertex graphs}
% \label{fig:four-vertex}
% \end{figure}

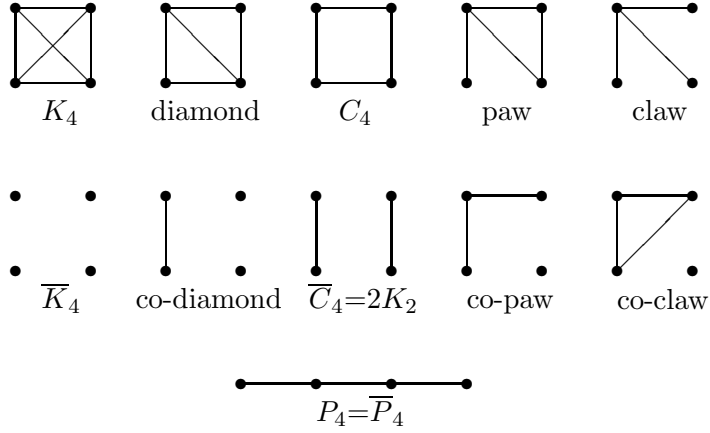
\begin{figure}[t]
	\begin{center}
		\begin{picture}(240,100)
		\setlength{\unitlength}{0.5mm}
		%\put(20,30){\circle*{3}} 
		\put(0,100){\circle*{3}}
		\put(0,80){\circle*{3}} 
		\put(20,100){\circle*{3}}
		\put(20,80){\circle*{3}} 
		
		\put(40,100){\circle*{3}}
		\put(40,80){\circle*{3}} 
		\put(60,100){\circle*{3}}
		\put(60,80){\circle*{3}} 
		
		\put(80,100){\circle*{3}}
		\put(80,80){\circle*{3}} 
		\put(100,100){\circle*{3}}
		\put(100,80){\circle*{3}}
		
		\put(120,100){\circle*{3}}
		\put(120,80){\circle*{3}} 
		\put(140,100){\circle*{3}}
		\put(140,80){\circle*{3}} 
		
		\put(160,100){\circle*{3}}
		\put(160,80){\circle*{3}} 
		\put(180,100){\circle*{3}}
		\put(180,80){\circle*{3}}

		\put(0,50){\circle*{3}}
		\put(0,30){\circle*{3}} 
		\put(20,50){\circle*{3}}
		\put(20,30){\circle*{3}} 
		
		\put(40,50){\circle*{3}}
		\put(40,30){\circle*{3}} 
		\put(60,50){\circle*{3}}
		\put(60,30){\circle*{3}} 
		
		\put(80,50){\circle*{3}}
		\put(80,30){\circle*{3}} 
		\put(100,50){\circle*{3}}
		\put(100,30){\circle*{3}}
		
		\put(120,50){\circle*{3}}
		\put(120,30){\circle*{3}} 
		\put(140,50){\circle*{3}}
		\put(140,30){\circle*{3}} 
		
		\put(160,50){\circle*{3}}
		\put(160,30){\circle*{3}} 
		\put(180,50){\circle*{3}}
		\put(180,30){\circle*{3}}

		\put(60,0){\circle*{3}}
		\put(80,0){\circle*{3}} 
		\put(100,0){\circle*{3}}
		\put(120,0){\circle*{3}}

		\put(60,0){\line(1,0){20}} 
		\put(80,0){\line(1,0){20}}
		\put(100,0){\line(1,0){20}}
		\put(80,-10){$P_{4}$=$\overline{P}_4$}
		%\put(76,-10){$P_{4}$=co-$P_4$}

		\put(0,100){\line(1,0){20}}
		\put(0,100){\line(0,-1){20}}
		\put(0,100){\line(1,-1){20}}
		\put(20,80){\line(0,1){20}}
		\put(20,80){\line(-1,0){20}}
		\put(0,80){\line(1,1){20}}
		\put(7,70){$K_{4}$}

		\put(40,100){\line(1,0){20}}
		\put(40,100){\line(0,-1){20}}
		\put(40,100){\line(1,-1){20}}
		\put(60,80){\line(0,1){20}}
		\put(60,80){\line(-1,0){20}}
		\put(36,70){diamond}
		
		\put(80,100){\line(1,0){20}}
		\put(80,100){\line(0,-1){20}}
		\put(100,80){\line(0,1){20}}
		\put(100,80){\line(-1,0){20}}
		\put(86,70){$C_{4}$}

		\put(120,100){\line(1,0){20}}
		\put(120,100){\line(0,-1){20}}
		\put(120,100){\line(1,-1){20}}
		\put(140,80){\line(0,1){20}}
		%\put(60,80){\line(-1,0){20}}
		\put(124,70){paw}

		\put(160,100){\line(1,0){20}}
		\put(160,100){\line(0,-1){20}}
		\put(160,100){\line(1,-1){20}}
		%\put(140,80){\line(0,1){20}}
		%\put(60,80){\line(-1,0){20}}
		\put(164,70){claw}

		\put(7,20){$\overline{K}_4$}
		%\put(2,20){co-$K_4$}

		\put(40,50){\line(0,-1){20}}
		\put(32,20){co-diamond}
		
		\put(80,50){\line(0,-1){20}}
		\put(100,50){\line(0,-1){20}}
		\put(78,20){$\overline{C}_4$=$2K_2$}
		%\put(75,20){co-$C_4$=$2K_2$}

		\put(120,50){\line(0,-1){20}}
		\put(120,50){\line(1,0){20}}
		\put(120,20){co-paw}
		
		\put(160,50){\line(0,-1){20}}
		\put(160,50){\line(1,0){20}}
		\put(160,30){\line(1,1){20}}
		\put(160,20){co-claw}
		\end{picture}
	\end{center}
	\caption{All 4-vertex graphs} 
	\label{fig:four-vertex}
\end{figure}

Below we list which of these graphs are $P_3$-free and which of them are forests (take the complements for  $\overline{P}_3$-free graphs and for the complements of forests, respectively). 
\begin{itemize}
	\item $P_3$-free graphs: $K_4$, $\overline{K}_4$,  $2K_2$, co-diamond, co-claw.
	\item Forests: $\overline{K}_4$, $2K_2$, $P_4$, co-diamond, co-paw, claw.
\end{itemize}

%%%%%%%%%%%%%%%%%%%%%%%%%%%%%%%%%%%%%%%%%%%%%%%%%%%%%%%%%%%%%%%%
%%%%%%%%%%%%%%%%%%%%%%%%%%%%%%%%%%%%%%%%%%%%%%%%%%%%%%%%%%%%%%%%
%%%%%%%%%%%%%%%%%%%%%%%%%%%%%%%%%%%%%%%%%%%%%%%%%%%%%%%%%%%%%%%%
%%%%%%%%%%%%%%%%%%%%%%%%%%%%%%%%%%%%%%%%%%%%%%%%%%%%%%%%%%%%%%%%
%%%%%%%%%%%%%%%%%%%%%%%%%%%%%%%%%%%%%%%%%%%%%%%%%%%%%%%%%%%%%%%%

\subsection{Claw- and co-claw-free graphs}
\label{sec:claw}

%%%%%%%%%%%%%%%%%%%%%%%%%%%%%%%%%%%%%%%%%%%%%%%%%%%%%%%%%%%%%%%%
%%%%%%%%%%%%%%%%%%%%%%%%%%%%%%%%%%%%%%%%%%%%%%%%%%%%%%%%%%%%%%%%
%%%%%%%%%%%%%%%%%%%%%%%%%%%%%%%%%%%%%%%%%%%%%%%%%%%%%%%%%%%%%%%%
%%%%%%%%%%%%%%%%%%%%%%%%%%%%%%%%%%%%%%%%%%%%%%%%%%%%%%%%%%%%%%%%
%%%%%%%%%%%%%%%%%%%%%%%%%%%%%%%%%%%%%%%%%%%%%%%%%%%%%%%%%%%%%%%%

\begin{lemma}\label{lem:k4}
	If a $($claw,co-claw$)$-free graph $G$ contains a $\overline{K}_4$, then it is $K_3$-free.
\end{lemma}

\begin{proof}
	Assume $G$ contains a $\overline{K}_4$ induced by $A=\{a_1,a_2,a_3,a_4\}$
	and suppose by contradiction that $G$ also contains a $K_3$ induced by $Z=\{x,y,z\}$.  
	
	Let first $A$ be disjoint from $Z$.
	To avoid a co-claw, each vertex of $A$ has a neighbour in $Z$ and hence one of the vertices of $Z$ is adjacent to two vertices of $A$, say $x$ is adjacent to $a_1$ and $a_2$.
	Then, to avoid a claw, $x$ has no other neighbours in $A$ and $y$ has a neighbour in $\{a_1,a_2\}$, say $y$ is adjacent to $a_1$. 
	This implies that $y$ is adjacent to $a_3$ (else $x,y,a_1,a_3$ induce a co-claw) and similarly $y$ is adjacent to $a_4$.
	But then $y,a_1,a_3,a_4$ induce a claw, a contradiction.
	
	If $A$ and $Z$ are not disjoint, they have at most one vertex in common, say $a_4=z$.
	Again, to avoid a co-claw, each vertex in $\{a_1,a_2,a_3\}$ has a neighbour in $\{x,y\}$ and hence, without loss of generality, $x$ is adjacent to $a_1$ and $a_2$.
	But then  $x,a_1,a_2,a_4$ induce a claw, a contradiction again.
	 \end{proof}

\begin{lemma}\label{lem:finite}
	The maximum number of vertices in a $($claw,co-claw,$K_4,\overline{K}_4)$-free graph  is 9.
\end{lemma}

\begin{proof}
	Let $G$ be a (claw,co-claw,$K_4,\overline{K}_4$)-free graph and let $x$ be a vertex of $G$. Denote by $A$ the set of neighbours and by $B$ the set of non-neighbours of $x$.
	Clearly, $A$ contains neither triangles nor anti-triangles, since otherwise either a $K_4$ or a claw arises.
	Therefore, $A$ has at most 5 vertices, and similarly $B$ has at most 5 vertices. 
	
	If $|A|=5$, then $G[A]$ must be a $C_5$ induced by vertices, say, $a_1,a_2,a_3,a_4,a_5$ (listed along the cycle).
	In order to avoid a claw or $K_4$, each vertex of $A$ can be adjacent to at most 2 vertices of $B$, which gives rise to at most 10 edges between $A$ and $B$.
	On the other hand, to avoid a co-claw, each vertex of $B$ must be adjacent to at least 3 vertices of $A$. Therefore, $B$ contains at most 3 vertices and hence 
	$|V(G)|\le 9$. Similarly, if $|B|=5$, then $|V(G)|\le 9$. 
	
%	It remains to show that there exists a (claw,co-claw,$K_4,\overline{K}_4$)-free graph with 9 vertices. This graph can be constructed as follows.
%	Start with a $C_8$ formed by the vertices $v_1,v_2,\ldots,v_8$. Then create a $C_4$ on the even-indexed vertices $v_2,v_4,v_6,v_8$ (listed along the cycle) and 
%	a $\overline{C}_4$ on the odd-indexed vertices $v_1,v_3,v_5,v_7$ (listed along the cycle in the complement).
%	Finally, add one more vertex adjacent to the odd-indexed vertices.  It is now a routine matter to check that the resulting graph is (claw,co-claw,$K_4,\overline{K}_4$)-free.
%	This graph is known as the Paley graph of order $q=3^2$, or the $3 \times 3$ rook's graph (see Figure~\ref{fig:paley} for an illustration).

	It remains to show that there exists a (claw,co-claw,$K_4,\overline{K}_4$)-free graph with 9 vertices. The $3 \times 3$ rook's graph (also known as the Paley graph of order $9$), shown in Figure~\ref{fig:paley}, is a witnessing example.
	\end{proof}

	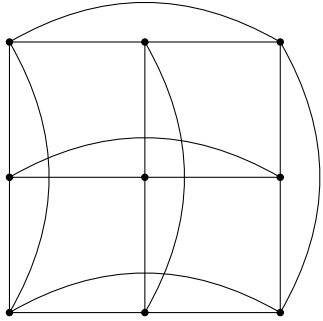
\begin{figure}[ht]
	\begin{center}
		\begin{tikzpicture}[scale=0.6, transform shape]
			
			\foreach \i in {0,...,2}
			{
				\foreach \x in {0,...,2}
				{
					\filldraw (3 * \i, 3 * \x) circle (2pt) node{}; 
				}
			
			%horizontal edges
			\draw (0, 3 * \i) -- (6, 3 * \i);
			\draw (0, 3 * \i) to [out=30,in=150] (6, 3 * \i);
			
			%vertical edges
			\draw (3 * \i, 0) -- (3 * \i, 6);
			\draw (3 * \i, 0) to [out=60,in=-60] (3 * \i, 6);
			}

%			\draw[dotted] (0, 0) -- (0, 7);
%			\draw[dotted] (0, 0) -- (7, 0);
			%	\draw[dotted] (-1, 0.93) -- (8, 0.93);
			
		\end{tikzpicture}
	\end{center}
	\caption{The $3 \times 3$ rook's graph.}
	\label{fig:paley}
\end{figure}

\begin{theorem}\label{thm:claw}
	For the class $A$ of $($claw,co-claw$)$-free graphs and all $a,b\ge 3$, 
	$$R_A(a,b)=\max \big\{ \left\lfloor(5a-3)/2\right\rfloor,\left\lfloor(5b-3)/2\right\rfloor \big\},$$
	unless $a=b=4$ in which case $R_A(a,b)=10$. 
\end{theorem}

\begin{proof}
	According to Lemma~\ref{lem:k4}, the class of  (claw,co-claw)-free graphs is the union of three classes: 
	\begin{itemize}
		\item the class $X$ of (claw,$K_3$)-free graphs, 
		\item the class $Y$ of (co-claw,$\overline{K}_3$)-free graphs and 
		\item the class $Z$ of (claw,co-claw,$K_4,\overline{K}_4$)-free graphs.
	\end{itemize}
	Clearly, $R_A(a,b)=\max \{ R_X(a,b),R_Y(a,b),R_Z(a,b) \}$.
	
	Since $K_3$ is forbidden in $X$, we have $R_X(a,b)=R_X(3,b)$,  Also, denoting by $B$ the class of claw-free graphs, we conclude that $R_X(3,b)=R_{B}(3,b)$. 
	As was shown in \cite{first}, $R_{B}(3,b)=\left\lfloor(5b-3)/2)\right\rfloor$. 
	Therefore,  $R_X(a,b)=\left\lfloor(5b-3)/2)\right\rfloor$. Similarly, $R_Y(a,b)=\left\lfloor(5a-3)/2)\right\rfloor$. 
	
	In the class $Z$, for all $a,b\ge 4$ we have  $R_Z(a,b)=10$  by Lemma~\ref{lem:finite}.
	Moreover, if additionally $\max \{ a,b \}\ge 5$, then $R_Z(a,b)<\max \{ R_X(a,b), R_Y(a,b) \}$.
	For $a=b=4$, we have $R_Z(4,4)=10>8=\max \{ R_X(4,4), R_Y(4,4) \}$.
	Finally, it is not difficult to see that $R_Z(3,b)\le R_X(3,b)$ and $R_Z(a,3)\le R_Y(a,3)$, and hence the result follows. 
	 \end{proof}

%%%%%%%%%%%%%%%%%%%%%%%%%%%%%%%%%%%%%%%%%%%%%%%%%%%%%%%%%%%%%%%%%%%%%%%%%%%%%%%%
%%%%%%%%%%%%%%%%%%%%%%%%%%%%%%%%%%%%%%%%%%%%%%%%%%%%%%%%%%%%%%%%%%%%%%%%%%%%%%%%
%%%%%%%%%%%%%%%%%%%%%%%%%%%%%%%%%%%%%%%%%%%%%%%%%%%%%%%%%%%%%%%%%%%%%%%%%%%%%%%%

\subsection{Diamond- and co-diamond-free graphs}

%%%%%%%%%%%%%%%%%%%%%%%%%%%%%%%%%%%%%%%%%%%%%%%%%%%%%%%%%%%%%%%%%%%%%%%%%%%%%%%%
%%%%%%%%%%%%%%%%%%%%%%%%%%%%%%%%%%%%%%%%%%%%%%%%%%%%%%%%%%%%%%%%%%%%%%%%%%%%%%%%
%%%%%%%%%%%%%%%%%%%%%%%%%%%%%%%%%%%%%%%%%%%%%%%%%%%%%%%%%%%%%%%%%%%%%%%%%%%%%%%%

\begin{lemma}\label{lem:k4-d}
	If a $($diamond,co-diamond$)$-free graph $G$ contains a $\overline{K}_4$, then it is bipartite. 
\end{lemma}

\begin{proof}
	Assume $G$ contains a $\overline{K}_4$. Let $A$ be any maximal (with respect to inclusion) independent set containing the $\overline{K}_4$ and let $B=V(G)-A$.
	If $B$ is empty, then $G$ is edgeless (and hence bipartite). Suppose now $B$ contains a vertex $b$. Then $b$ has a neighbour $a$ in $A$ (else $A$ is not maximal) and at most one
	non-neighbour (else $a$ and $b$ together with any two non-neighbours of $b$ in $A$ induce a co-diamond). 
	
	Assume $B$ has two adjacent vertices, say $b_1$ and $b_2$. Since $|A|\ge 4$ and each of $b_1$ and $b_2$ has at most one non-neighbour in $A$, there are
	at least two common neighbours of $b_1$ and $b_2$ in $A$, say $a_1,a_2$. But then $a_1,a_2,b_1,b_2$ induce a diamond. This contradiction shows that $B$ is independent and hence 
	$G$ is bipartite. 
	 \end{proof}

\begin{lemma}\label{lem:cd-bipartite}
	A co-diamond-free bipartite graph containing at least one edge is either a simplex (a bipartite graph in which every vertex has at most one non-neighbour in the opposite part) 
	or a $K_{s,t}+K_1$ for some $s$ and $t$.
\end{lemma}

\begin{proof}
	Assume $G=(A,B,E)$ is a co-diamond-free bipartite graph containing at least one edge. Then $G$ cannot have two isolated vertices, since otherwise
	an edge together with two isolated vertices create an induced co-diamond.
	
	Assume $G$ has exactly one isolated vertex, say $a$, and let $G'=G-a$ Then any vertex $b\in V(G')$ is adjacent to every vertex in the opposite part of $G'$. 
	Indeed, if $b$ has a non-neighbour $c$ in the opposite part, then $a,b,c$ together with any neighbour of $b$ (which exists because $b$ is not isolated) induce a co-diamond. 
	Therefore, $G'$ is complete bipartite and hence $G=K_{s,t}+K_1$ for some $s$ and $t$.  
	
	Finally, suppose $G$ has no isolated vertices. Then every vertex $a\in A$ has at most one non-neighbour in $B$, since otherwise any two non-neighbours of $a$ in $B$ 
	together with $a$ and any neighbour of $a$ (which exists because $a$ is not isolated) induce a co-diamond. Similarly, every vertex $b\in B$ has at most one non-neighbour in $A$.
	Therefore, $G$ is a simplex.
	 \end{proof}

\begin{lemma}\label{lem:finite-d}
	The maximum number of vertices in a $($diamond,co-diamond,$K_4,\overline{K}_4)$-free graph  is 9.
\end{lemma}

\begin{proof}
	Let $G$ be a (diamond,co-diamond,$K_4,\overline{K}_4$)-free graph and $x$ be a vertex of $G$. Denote by $A$ the set of neighbours and by $B$ the set of non-neighbours of $x$.
	Then $G[A]$ is $(P_3,K_3)$-free, else $G$ contains either a diamond or a $K_4$. Since $G[A]$ is $P_3$-free, every connected component of $G[A]$ is a clique and since this graph is $K_3$-free,
	every connected component has at most 2 vertices. If at least one of the components of $G[A]$ has 2 vertices, the number of components is at most 2 (since otherwise a co-diamond arises),
	in which case $A$ has at most 4 vertices. If all the components of $G[A]$ have size 1, the number of components is at most 3 (since otherwise a $\overline{K}_4$  arises),
	in which case $A$ has at most 3 vertices. Similarly, $B$ has at most 4 vertices and hence $|V(G)|\le 9$. 
	
	To conclude the proof, we observe that the Paley graph of order $q=3^2$ described in the proof of Lemma~\ref{lem:finite} is (diamond,co-diamond,$K_4,\overline{K}_4$)-free.
	 \end{proof}

\begin{theorem}\label{thm:diamond}
	For the class $A$ of $($diamond,co-diamond$)$-free graphs and $a,b\ge 3$, 
	$$R_A(a,b)=\max \{ 2a-1,2b-1 \},$$
	unless $a,b\in\{4,5\}$, in which case $R_A(a,b)=10$, and unless $a=b=3$, in which case $R_A(a,b)=6$. 
\end{theorem}

\begin{proof}
	According to Lemma~\ref{lem:k4-d}, in order to determine the value of $R_A(a,b)$, we analyze this number in three classes: 
	\begin{itemize}
		\item the class $X$ of co-diamond-free bipartite graphs, 
		\item the class $Y$ of the complements of graphs in $X$ and 
		\item the class $Z$ of (diamond,co-diamond,$K_4,\overline{K}_4$)-free graphs.
	\end{itemize}
	In the class $X$ of co-diamond-free bipartite graphs, $R_X(a,b)=2b-1$, since every graph in this class with at least $2b-1$ contains an independent set of size $b$, while the graph $K_{b-1,b-1}$
	contains neither an independent set of size $b$ nor a clique of size $a\ge 3$. Similarly, $R_{Y}(a,b)=2a-1$.
	
	In the class $Z$ of (diamond,co-diamond,$K_4,\overline{K}_4$)-free graphs, for all $a,b\ge 4$ we have  $R_Z(a,b)=10$  by Lemma~\ref{lem:finite-d}.
	Moreover, if additionally $\max \{a,b\}\ge 6$, then $R_Z(a,b)<\max \{ R_X(a,b), R_{Y}(a,b) \}$.
	For $a,b\in\{4,5\}$, we have $R_Z(a,b)=10>\max \{ R_X(a,b), R_Y(a,b) \}$. Also, $R_Z(3,3)=6$ (since $C_5\in Z$) and hence $R_Z(3,3)>\max \{ R_X(3,3), R_Y(3,3) \}$.
	Finally, by direct inspection one can verify that $Z$ contains no $K_3$-free graphs with more than 6 vertices and hence for $b\ge 4$ we have $R_Z(3,b)\le R_X(3,b)$.
	Similarly, for $a\ge 4$ we have $R_Z(a,3)\le R_{Y}(a,3)$. Thus for all values of $a,b\ge 3$, we have $R_A(a,b)=\max \{ 2a-1,2b-1 \}$, 
	unless $a,b\in\{4,5\}$, in which case $R_A(a,b)=10$, and unless $a=b=3$, in which case $R_A(a,b)=6$. 
	 \end{proof}

%%%%%%%%%%%%%%%%%%%%%%%%%%%%%%%%%%%%%%%%%%%%%%%%%%%%%%%%%%%%%%%%%%%%%%%%%%%%%%%%
%%%%%%%%%%%%%%%%%%%%%%%%%%%%%%%%%%%%%%%%%%%%%%%%%%%%%%%%%%%%%%%%%%%%%%%%%%%%%%%%
%%%%%%%%%%%%%%%%%%%%%%%%%%%%%%%%%%%%%%%%%%%%%%%%%%%%%%%%%%%%%%%%%%%%%%%%%%%%%%%%

\subsection{$2K_2$- and $C_4$-free graphs}

%%%%%%%%%%%%%%%%%%%%%%%%%%%%%%%%%%%%%%%%%%%%%%%%%%%%%%%%%%%%%%%%%%%%%%%%%%%%%%%%
%%%%%%%%%%%%%%%%%%%%%%%%%%%%%%%%%%%%%%%%%%%%%%%%%%%%%%%%%%%%%%%%%%%%%%%%%%%%%%%%
%%%%%%%%%%%%%%%%%%%%%%%%%%%%%%%%%%%%%%%%%%%%%%%%%%%%%%%%%%%%%%%%%%%%%%%%%%%%%%%%

\begin{theorem}\label{thm:2k2}
	For the class $A$ of $(2K_2,C_4)$-free graphs and all $a,b\ge 3$, 
	$$R_A(a,b)=a+b.$$
\end{theorem}

\begin{proof}
	Let $G$ be a ($2K_2,C_4$)-free graph with $a+b$ vertices. If, in addition, $G$ is $C_5$-free, then the three forbidden induced subgraphs ensures that $G$ belongs to the class of split graphs and hence it contains either 
	a clique of size $a$ or an independent set of size $b$. 
	
	If $G$ contains a $C_5$, then the remaining vertices of the graph can be partitioned into a clique  $U$, whose vertices are complete to the cycle $C_5$, 
	and an independent set $W$, whose vertices are anticomplete to the $C_5$ \cite{pseudo-split}.  We have $|U|+|W|=a+b-5$ and hence either $|U|\ge a-2$ or $|W|\ge b-2$. In the first case, 
	$U$ together with any two adjacent vertices of the cycle $C_5$ create a clique of size $a$. In the second case, $W$ together with any two non-adjacent vertices of 
	the cycle create an independent set of size $b$. This shows that  $R_A(a,b)\le a+b$.
	
	For the inverse inequality, we construct a graph $G$ with $a+b-1$ vertices as follows: $G$ consists of a cycle $C_5$, an independent set $W$ of size $b-3$ anticomplete to the cycle
	and a clique $U$ of size $a-3$ complete to both $W$ and $V(C_5)$. It is not difficult to see that the size of a maximum clique in $G$ is $a-1$ and the size of a maximum independent set  in $G$ is $b-1$.
	Therefore, $R_A(a,b)\ge a+b$.
	 \end{proof}

%%%%%%%%%%%%%%%%%%%%%%%%%%%%%%%%%%%%%%%%%%%%%%%%%%%%%%%%%%%%%%%%%%%%%%%%%%%%%%%%
%%%%%%%%%%%%%%%%%%%%%%%%%%%%%%%%%%%%%%%%%%%%%%%%%%%%%%%%%%%%%%%%%%%%%%%%%%%%%%%%
%%%%%%%%%%%%%%%%%%%%%%%%%%%%%%%%%%%%%%%%%%%%%%%%%%%%%%%%%%%%%%%%%%%%%%%%%%%%%%%%

\subsection{$2K_2$- and diamond-free graphs}

%%%%%%%%%%%%%%%%%%%%%%%%%%%%%%%%%%%%%%%%%%%%%%%%%%%%%%%%%%%%%%%%%%%%%%%%%%%%%%%%
%%%%%%%%%%%%%%%%%%%%%%%%%%%%%%%%%%%%%%%%%%%%%%%%%%%%%%%%%%%%%%%%%%%%%%%%%%%%%%%%
%%%%%%%%%%%%%%%%%%%%%%%%%%%%%%%%%%%%%%%%%%%%%%%%%%%%%%%%%%%%%%%%%%%%%%%%%%%%%%%%
To analyze this class, we split it into three subclasses $X, Y$ and $Z$ as follows: 
\begin{itemize}
\item[$X$] is the class of $(2K_2$,diamond)-free graphs containing a $K_4$,
\item[$Y$] is the class of $(2K_2$,diamond)-free graphs that do not contain a $K_4$ but contain a $K_3$, 
\item[$Z$] is the class of $(2K_2$,diamond)-free graphs that do not contain a $K_3$, i.e., the class of $(2K_2,K_3)$-free graphs. 
\end{itemize}
We start by characterizing graphs in the class $X$.

\begin{lemma}
	If a $(2K_2$,diamond$)$-free graph $G$ contains a $K_4$, then $G$ is a split graph 
	partitionable into a clique $C$ and an independent set $I$ such that every vertex of $I$ has at most one neighbour in $C$.
\end{lemma}

\begin{proof}
	Let $G$ be a $(2K_2$,diamond)-free graph containing a $K_4$. We extend the $K_4$ to any maximal (with respect to inclusion) clique and denote it by $C$.
	Also, denote $I=V(G)-C$.
	
	Assume a vertex $a\in I$ has two neighbours $b,c$ in $C$. It also has a non-neighbour $d$ in $C$ (else $C$ is not maximal). But then $a,b,c,d$
	induce a diamond. This contradiction shows that any vertex of $I$ has at most one neighbour in $C$. 
	
	Finally, assume two vertices $a,b\in I$ are adjacent. Since each of them has at most one neighbour in $C$ and $|C|\ge 4$, there are two vertices $c,d\in C$
	adjacent neither to $a$ nor to $b$. But then $a,b,c,d$ induce a $2K_2$. This contradiction shows that $I$ is independent and completes the proof.  
	 \end{proof}

%Next, we characterize graphs in the class $Z$, since $Y$ is a tricker class and requires more more work.
% \begin{lemma}
% 	Let $G$ be a $(2K_2,diamond,K_4)$-free graph  containing a $K_3$. Then $G$ is 3-colorable. 
% \end{lemma}

% \begin{proof}
% 	Denote  a triangle $K_3$ in $G$ by $T=\{a,b,c\}$, and for any subset $U\subseteq \{a,b,c\}$ let $V_U$ be the subset of vertices outside of $T$ such that $N(v)\cap T=U$ for each $v\in V_U$.
% 	Then 
% 	\begin{itemize}
% 		\item $V_{a,b,c}=\emptyset$, since $G$ is $K_4$-free.
% 		\item $V_{a,b}=V_{ac}=V_{bc}=\emptyset$, since $G$ is diamond-free.
% 		\item $V_a,V_b,V_c,V_\emptyset$ are independent sets, since $G$ is $2K_2$-free. For the same reason, every vertex of $V_\emptyset$ is isolated. 
% 	\end{itemize} 
% 	Then each of the following three sets $\{a\}\cup V_b$, $\{b\}\cup V_c$ and  $\{c\}\cup V_a\cup V_\emptyset$ is independent and hence $G$ is  3-colorable. 
% 	 \end{proof}

% The above two lemmas reduce the analysis to $(2K_2,K_3)$-free graphs. 
In order to characterize graphs in $Z$, let us say that $G^*$ is an extended $G$ (also known as a blow-up of $G$) if  $G^*$ is obtained from $G$ by replacing the vertices of $G$ with independent sets.

\begin{lemma} \label{lem:2K2-K3}
	If $G$ is a $(2K_2,K_3)$-free graph, then it is either bipartite or an extended $C_5+K_1$.
\end{lemma}

\begin{proof}
	If $G$ is $C_5$-free, then it is bipartite, because any cycle of length at least 7 contains an induced $2K_2$. Assume now that $G$ contains a $C_5$ induced by a set $S=\{v_0,v_1,v_2,v_3,v_4\}$.
	To avoid an induced $2K_2$ or $K_3$, any vertex $u\not\in S$ must be either anticomplete to $S$ or have exactly two neighbours on the cycle of distance 2 from each other,
	%VZ: modified below
	i.e., $N(u)\cap S=\{v_{i-1},v_{i+1}\}$ for some $i$ (addition is taken modulo 5). Moreover, if $N(u)\cap S=\{v_{i-1},v_{i+1}\}$ and $N(w)\cap S=\{v_{j-1},v_{j+1}\}$, then 
%	i.e., $N(u)\cap S=\{v_i,v_{i+2}\}$ for some $i$ (addition is taken modulo 5). Moreover, if $N(u)\cap S=\{v_i,v_{i+2}\}$ and $N(w)\cap S=\{v_j,v_{j+2}\}$, then 
	\begin{itemize}
		\item if $i=j$ or $|i-j|>1$, then $u$ is not adjacent to $w$, since $G$ is $K_3$-free.
		\item if $|i-j|=1$, then $u$ is adjacent to $w$, since $G$ is $2K_2$-free.
	\end{itemize}
	Clearly, every vertex $u\not\in S$, which is anticomplete to $S$, is isolated, and hence $G$ is an extended $C_5+K_1$.
	 \end{proof}

Now we turn to graphs $G$ in the class $Y$ and characterize them through a series of claims.

\begin{itemize}
\item[(1)] 
%\label{disjtri} 
{\it Any two triangles in $G$ are vertex disjoint}. To see this, note that two triangles intersecting in two vertices 
induce either a $K_4$ or a diamond. If two triangles induced by say $x_1, y_1, z$ and $x_2, y_2, z$ intersect in a single vertex, 
there must be another edge between them, say $x_1x_2$, since otherwise we obtain an induced $2K_2$. But then $x_1, x_2, y_1, z$ induce two triangles intersecting in two vertices.
		
\item[(2)] 
%\label{edgetriangle} 
{\it For any edge $xy$ and a triangle $T$ containing neither $x$ nor $y$, $x$ and $y$ each have exactly one neighbour in $T$}. 
Indeed, $x$ and $y$ each have at most one neighbour, since otherwise we obtain two triangles intersecting in two vertices. 
Moreover, if one of them does not have a neighbour, an induced $2K_2$ appears. We note that the neighbours of $x$ and $y$ must be distinct, since otherwise we obtain two triangles intersecting in one vertex. It follows, in particular, that the edges between two triangles form a matching. 
		
%\item[(3)] 
%{\it $G$ contains at most 3 triangles}. To see this, suppose for a contradiction that $a_i, b_i, c_i, 0 \leq i \leq 3$ induce 4 triangles. 
%Let us study how this 4 triangle configuration looks like: by symmetry, we may assume that $a_0$ is connected to $a_1, a_2$ and $a_3$, and similarly with $b_0$ and $c_0$. 
%The configuration is then determined by the matchings between triangles 1, 2 and 3. Since triangles in $G$ are vertex disjoint, 
%we know that there are no other edges between vertices with the same letter. In other words, each of the three matchings that we will denote by 
%$1 \to 2$, $1 \to 3$ and $2 \to 3$ can be described by a permutation of $\{a, b, c\}$ with no fixed points. 
%The only two such permutations are the two cyclic ones $(a, b, c)$ and $(a, c, b)$. 
%Again one can check that by symmetry, we only need to consider the case where all three permutations are described by $(a, b, c)$. 
%But then $a_0, a_2, b_1, c_3$ induce a $2K_2$. 
		
\item[(3)] 
%\label{noC5} 
{\it If $G$ has a triangle $T$, it does not contain an induced $C_5$ vertex disjoint from $T$}. 
To see this, assume that $G$ has a triangle $x, y, z$ and a $C_5$ induced by $v_1, v_2, v_3, v_4, v_5$. 
By (2), each vertex in the $C_5$ has exactly one neighbour in the triangle, and no two consecutive $v_i$ (modulo 5) have the same neighbour in the triangle. 
It follows that up to isomorphism, the edges between the triangle and the $C_5$ are $xv_1$, $yv_2$, $yv_4$, $zv_3$, $zv_5$. But then $x, v_1, v_3, v_4$ induce a $2K_2$.

\item[(4)]
%\label{3Tisolated} 
{\it If $G$ contains 3 triangles $T_i$, each induced by $a_i, b_i, c_i, 1 \leq i \leq 3$, then every other vertex in the graph is isolated. In particular, $G$ contains at most $3$ triangles.}
Without loss of generality, using Claim~(1) and by symmetry, the edges between the triangles are given by $a_ib_j$, $b_ic_j$, $c_ia_j$ with $i \leq j$. 
Suppose for a contradiction that $x$ is a non-isolated vertex not in the $T_i$. Then $x$ has exactly one neighbour in each of the triangles. 
Indeed, by Claim~(1), it has at most one neighbour in each triangle, and if it has a neighbour anywhere in the graph, Claim~(2) applies.  
Without loss of generality, suppose the neighbour of $x$ in $T_2$ is $b_2$. Then $x$ must be adjacent to exactly one of $a_3$ and $c_1$, since otherwise $x, b_2, a_3, c_1$ induce a $2K_2$. 
If $x$ is adjacent to $a_3$, then $x, b_2, a_2, b_3, a_3$ induce a $C_5$ vertex disjoint from $T_1$, contrary to Claim~(3). 
Similarly, if $x$ is adjacent to $c_1$, then $x, b_2, c_2, b_1, c_1$ induce a $C_5$ vertex disjoint from $T_3$.

\item[(5)]
{\it If $G$ contains exactly 2 triangles $T_1$ and $T_2$ 
and the graph $G'=G-(T_1\cup T_2)$
contains an edge, then $G'$ admits a bipartition $X' \cup Y'$ such that there exist vertices $z_1 \in T_1$ and $z_2 \in T_2$ with 
the property that $X' \cup \{z_1, z_2\}$ and $Y' \cup \{z_1, z_2\}$ are independent sets}. Note $G' \in Z$, and by Claim~(4), it is $C_5$-free. 
It follows from Lemma~\ref{lem:2K2-K3} that $G'$ is a $2K_2$-free bipartite graph. We further split $G'$ into $G'_1$ and $G'_0$, where $G'_0$ consists of the isolated vertices in $G'$, while $G'_1$ contains the rest of the vertices of $G'$.

Note that, since $G$ is $2K_2$-free, $G'_1$ is a connected graph. As this graph contains an edge, by Claim~(2), every vertex of $G'_1$ has exactly one neighbour in each of $T_1$ and $T_2$. By standard structural results on bipartite $2K_2$-free graphs (also known as ``bipartite chain graphs''), each part of $G'_1$ has a dominating vertex, i.e., a vertex adjacent to all the vertices in the opposite part. 
Write $x$ and $y$ for those dominating vertices, and call their respective parts $X''$ and $Y''$. 
Let $y_1$ and $y_2$ be the neighbours of $x$ in $T_1$ and $T_2$ respectively, and similarly let $x_1$ and $x_2$ be the neighbours of $y$ in those triangles. 
By Claim~(1), $x_1 \neq y_1$, $x_2 \neq y_2$, $x_1$ and $x_2$ are not adjacent, and $y_1$ and $y_2$ are also not adjacent. Finally, write $z_1$, $z_2$ for the remaining two vertices in $T_1$ and $T_2$, respectively, and note that $z_1$ and $z_2$ are also not adjacent (otherwise $z_1z_2$ and $xy$ induce a $2K_2$). 
		
Note that any vertex in $X''$ must be adjacent to $y_1$ and to $y_2$: indeed, if for instance $x' \in X''$ is adjacent to $y'_1 \neq y_1$ in $T_1$, 
then $y$ is adjacent to neither of $y_1$ and $y'_1$ (by Claim~(1)), and so $x, x', y, y_1, y'_1$ induce a $C_5$ disjoint from $T_2$, contrary to Claim~(3). 
Similarly, every vertex in $Y''$ is adjacent to $x_1$ and to $x_2$.
		
It remains to deal with the vertices in $G'_0$. Let $w$ be a vertex in $G'_0$. Note that $w$ is non-adjacent to both $z_1$ and $z_2$, since any such edge together with $xy$ would induce a $2K_2$. Then  $X'' \cup G'_0  \cup \{z_1, z_2\}$ and $Y'' \cup \{z_1, z_2\}$ are independent sets as claimed.
%We look at two sub-cases:
%
%\begin{itemize}
%			\item $z_1$ and $z_2$ are adjacent. In this case, the other edges between the two triangles are $x_1y_2$ and $x_2y_1$. In particular, again by  Claim~(1), 
%$w$ cannot have neighbours both in $\{x_1,x_2\}$ and in $\{y_1,y_2\}$.  Write $G'_{0, x}$ for the vertices in $G'_0$ non-adjacent to both $x_1$ and $x_2$, 
%and $G'_{0, y}$ for the vertices in $G'_0$ non-adjacent to both $y_1$ and $y_2$ (breaking ties arbitrarily). 
%Then $X'' \cup G'_{0, x} \cup \{x_1, x_2\}$  and $Y'' \cup G'_{0, y} \cup  \{y_1, y_2\} $ are independent sets as required.
%			
%			\item $z_1$ and $z_2$ are non-adjacent. Note that $w$ is non-adjacent to both $z_1$ and $z_2$, since any such edge together with $xy$ would induce a $2K_2$. 
%Then  $X'' \cup G'_0  \cup \{z_1, z_2\}$ and $Y'' \cup \{z_1, z_2\}$ are again independent sets as required.
%
%\end{itemize}
\end{itemize}

\begin{theorem}
	Let  $A$ be the class of $(2K_2$,diamond$)$-free graphs. Then  
	\begin{itemize}
		\item for $a=3$, we have $R_A(a,b)=\lfloor 2.5(b-1)\rfloor+1$;
		\item for $a=4$, we have $R_A(a, 3) = 7$, $R_A(a, 4) = 10$ and $R_A(a,b)=\lfloor 2.5(b-1)\rfloor+1$ for $b \geq 5$;
		\item for $a\ge 5$, we have $R_A(a,b) = \max \{ \lfloor 2.5(b-1)\rfloor+1, a + b - 1 \}$, except for $R_A(5, 4) = 10$.  
	\end{itemize} 
\end{theorem}

\begin{proof} As before, we split the analysis into several subclasses of $A$.
	
	For the class $X$ of $(2K_2$,diamond)-free graphs containing a $K_4$ and $a\ge 5$, we have $R_X(a,b)=a+b-1$. Indeed, 
	every split graph with $a+b-1$ vertices contains either a clique of size $a$ or an independent set of size $b$ and hence $R_X(a,b)\le a+b-1$.
	On the other hand, the split graph with a clique $C$ of size $a-1$ and an independent set $I$ of size $b-1$ with a matching between $C$ and $I$ belongs to $X$ and hence   
	$R_X(a,b)\ge a+b-1$.
	
	%	For the class $Y$ of 3-colorable $(2K_2$,diamond)-free graphs and for $a\ge 4$ we have $R_Y(a,b)=3b-2$. Indeed, a 3-colorable graph with $3b-2$ vertices contains an independent set of size $b$
	%	and hence $R_Y\le 3b-2$. On the other hand, consider the graph $G$ constructed from $b-1$ triangles $T_i=\{a_i,b_i,c_i\}$ $(i=1,2,\ldots,b-1)$ such that for all $j>i$, 
	%	\begin{itemize} 
	%		\item $a_i$ is adjacent to $b_j$,
	%		\item $b_i$ is adjacent to $c_j$, 
	%		\item $c_i$ is adjacent to $a_j$.
	%	\end{itemize}
	%	It is not difficult to see that $G$ is  3-colorable $(2K_2$,diamond)-free graph with $3b-3$ vertices containing neither a clique of size   $a\ge 4$ nor an independent set of size $b$.
	%	Therefore, $R_Y\ge 3b-2$.  
	
	For the class $Z_0$ of bipartite $2K_2$-free graphs, we have $R_{Z_0}(a,b)=2b-1$, which is easy to see. 
For the class $Z_1$ of graphs each of which is an extended $C_5+K_1$, we have $R_{Z_1}(a,b)=\lfloor 2.5(b-1)\rfloor+1$.
	For an odd $b$, 
	%VZ: modified below
	an extremal graph
	%a maximum counterexample 
	is constructed from a $C_5$ by replacing each vertex with an independent set of size $(b-1)/2$.
	This graph has $\lfloor 2.5(b-1)\rfloor$ vertices, the independence number $b-1$ and the clique number $2<a$.
	For an even $b$, 
	%VZ: modified below
	an extremal graph
	%a maximum counterexample 
	is constructed from a $C_5$ by replacing two adjacent vertices of a $C_5$ 
	with independent sets of size $b/2$ and the remaining vertices of the cycle with independent sets of size $b/2-1$.
	This again gives in total $\lfloor 2.5(b-1)\rfloor$ vertices, and the independence number $b-1$. 
	Therefore, in the class $Z=Z_0\cup Z_1$, we have $R_{Z}(a,b)=\max \{ R_{Z_0}(a,b),R_{Z_1}(a,b) \}=\lfloor 2.5(b-1)\rfloor+1$. 
	
	To compute $R_Y(a, b)$, we partition $Y$ into $Y_1$, $Y_2$ and $Y_3$, where $Y_s$ consists of the graphs in $Y$ with $s$ triangles. We then have:
	
	\begin{itemize}
		\item $R_{Y_3}(4, b) = b + 6$ for $b \geq 4$.  Indeed, the three triangle configuration (unique up to isomorphism) has independence number 3, and any additional vertices are isolated by Claim~(4). 
		
		\item $R_{Y_2}(4, b) = 2b + 1$ for $b \geq 3$. To show this, let $G \in Y_2$ be a graph on $2b + 1$ vertices, with triangles $T_1$ and $T_2$. 
As in Claim~(5), $G'=G-(T_1\cup T_2)$, $G'_0$ consists of the isolated vertices in $G'$, while $G'_1$ contains the rest of $G'$.	

If $G'_1$ is empty (or in other words, if $G'$ has no edges), then $G' = G'_0$ is an independent set with $2b + 1 - 6 = 2b - 5$ vertices. Provided $b \geq 5$, this number is at least $b$. 
For $b = 3$, the unique vertex in $G'$ has at most one neighbour in each of $T_1$ and $T_2$, so in particular, it has two non-adjacent non-neighbours in the triangles, hence $G$ has an independent set of size 3. 
For $b = 4$, there are 3 vertices in $G'$. Like before, each of them has at most one neighbour in each triangle; if their neighbourhoods do not cover the triangles, 
then those three vertices together with a common non-neighbour give an independent set of size 4. If their neighbourhoods do cover the triangles, 
then by size constraints the neighbourhoods are disjoint, and each of them is an independent set by Claim~(1). 
In this case, any two of the vertices together with the neighbourhood of the third form an independent set of size 4.
		
Now assume that $G'$ has an edge. Then by Claim~(5) $G'$ admits a bipartition $X' \cup Y'$ such that there exist vertices $z_1 \in T_1$ and $z_2 \in T_2$ 
with the property that $X' \cup \{z_1, z_2\}$ and $Y' \cup \{z_1, z_2\}$ are independent sets. Given such a bipartition, it immediately follows 
that $G$ has an independent set of size at least $\left \lceil \frac{2b - 5}{2} \right \rceil + 2 = b$. 
		
Extremal counterexamples, i.e., graphs without clique of size 4 and without independent sets of size $b$, can be easily constructed, 
by making for instance $G'$ complete bipartite with $b - 3$ vertices in each part and connecting each part to the triangles appropriately.
		
\item $R_{Y_1}(4, b) \leq 2b + 1$ for $b \geq 3$. To see why, let $G \in Y_1$ be a graph on $2b + 1$ vertices, write $T$ for the triangle, and put $G' = G - T$. 
Like before, $G'$ is a $2K_2$-free bipartite graph; if it has isolated vertices, one can find a bipartition of $G'$ where one of the parts has size at least $b$. 
Otherwise, there are vertices $x$ and $y$ dominating each part. Those have neighbours $y'$ and $x'$ in $T$ respectively; but then by Claim~(1), 
$y'$ is a common non-neighbour of the part containing $y$, and $x'$ is a common non-neighbour of the part containing $x$. 
Since $G'$ has one part with size at least $b - 1$, this means $G$ contains an independent set of size $b$. 		
	\end{itemize}
Putting these together, we have $R_Y(4, b) = 2b + 1$ for $b \geq 3$, except $R_Y(4, 4) = 10$.

Combining the results for the three classes $X$, $Y$ and $Z$, we obtain the desired conclusion of the theorem.
\end{proof}

%%%%%%%%%%%%%%%%%%%%%%%%%%%%%%%%%%%%%%%%%%%%%%%%%%%%%%%%%%%%%%%%%%%%%%%%%%%%%%%%%%%%%%%%%%%%%%%%%%%%%%%%%
\subsection{The class of ($P_4,C_4$,co-claw)-free graphs}
%%%%%%%%%%%%%%%%%%%%%%%%%%%%%%%%%%%%%%%%%%%%%%%%%%%%%%%%%%%%%%%%%%%%%%%%%%%%%%%%%%%%%%%%%%%%%%%%%%%%%%%%%
% For the class $X$ of $P_4$-free graphs, the Ramsey number $R_X(,a,b)$ was determined in \cite{first) and is equal $(a-1,b-1)+1$.
% This is not a linear Ramsey number, which supports the ``only if'' part of our conjecture, because $P_4$ is neither $P_3$-free nor $\overline{P}_3$-free.

We start with a lemma characterizing the structure of graphs in this class, where we use the following well-known fact (see, e.g., \cite{cographs}):
every $P_4$-free graph with at least two vertices is either disconnected or the complement to a disconnected graph. 

%VZ: Rewrote the lemma below: it seems for our purposes we need a weeker statment, where a ``collection of stars'' is replaced by a ``bipartite graph''. 
\begin{lemma}\label{lem:P4C4coclaw}
	Every disconnected $(P_4,C_4$,co-claw$)$-free graph is a bipartite graph and every
	connected $(P_4,C_4$,co-claw$)$-free graph consists of a bipartite graph plus a number of dominating
	vertices, i.e., vertices adjacent to all other vertices of the graph.
\end{lemma}

\begin{proof}
	Let $G$ be a disconnected ($P_4,C_4$,co-claw)-free graph. Then every connected component of $G$ is $K_3$-free, 
	since a triangle in one of them together with a vertex from any other component create an induced co-claw. 
	Therefore, every connected component of $G$, and hence $G$ itself, is a bipartite graph (since forbidding $P_4$ forbids every cycle of length at least 5).
	
	Now let $G$ be a connected graph. Since $G$ is $P_4$-free, $\overline{G}$ is disconnected. 
	Let $C^1,\ldots,C^k$ $(k\ge 2)$ be co-components of $G$, i.e., components in the complement of $G$. 
	If at least two of them have more than 1 vertex, then an induced $C_4$ arises. 
	Therefore, all co-components, except possibly one, have size 1, i.e., they are dominating vertices in $G$.
	If, say, $C^1$ is a co-component of size more than 1, then the subgraph of $G$ induced by $C^1$ must be disconnected 
	and hence it is a bipartite graph.     
\end{proof}

%\begin{lemma}\label{lem:P4C4coclaw}
%	Every disconnected ($P_4,C_4$,co-claw)-free graph is a collection of disjoint stars and every
%	connected ($P_4,C_4$,co-claw)-free graph consists of a collection of disjoint stars plus a number of dominating vertices, i.e., vertices adjacent to all other vertices of the graph.
%\end{lemma}
%
%\begin{proof}
%	Let $G$ be a disconnected ($P_4,C_4$,co-claw)-free graph. Then every connected component of $G$ is $K_3$-free, 
%	since a triangle in one of them together with a vertex from any other component create an induced co-claw. 
%	Therefore, every connected component of $G$ is a bipartite graph. This graph is complete bipartite, because $G$ is $P_4$-free, and it is a star, because $G$ is $C_4$-free, 
%	
%	Now let $G$ be a connected graph. Since $G$ is $P_4$-free, $\overline{G}$ is disconnected. Let $C^1,\ldots,C^k$ $(k\ge 2)$ be co-components of $G$, i.e., components in the complement of $G$. 
%	If at least two of them have more than 1 vertex, then an induced $C_4$ arises. Therefore, all co-components, except possibly one, have size 1, i.e., they are dominating vertices in $G$.
%	If, say, $C^1$ is a co-component of size more than 1, then the subgraph of $G$ induced by $C^1$ must be disconnected and hence it is a collection of stars.     
%	 \end{proof}

\begin{theorem}
	For the class $A$ of ($P_4,C_4$,co-claw)-free graphs and all $a,b\ge 3$, 
	$$R_A(a,b)=a+2b-4.$$
\end{theorem}

\begin{proof}
	Let $G$ be a graph in $A$ with $a+2b-5$ vertices, $2b-2$ of which induce a matching (a 1-regular graph with $b-1$ edges) and the remaining $a-3$ vertices are dominating in $G$. 
	Then $G$ has neither a clique of size $a$ nor an independent set of size $b$. Therefore, $R_A(a,b)\ge a+2b-4.$
	
	Conversely, let $G$ be a graph in $A$ with  $a+2b-4$ vertices. If $G$ is disconnected, then, by Lemma \ref{lem:P4C4coclaw}, 
	it is bipartite and hence at least one part in a bipartition of $G$ has size at least $b$,
	i.e., $G$ contains an independent set of size $b$. If $G$ is connected, denote by $C$ the set of dominating vertices in $G$. If $|C|\ge a-1$, then 
	either $C$ itself (if $|C|\ge a$) or $C$ together with a vertex not in $C$ (if $|C|= a-1$) create a clique of size $a$. 
	So, assume $|C|\le a-2$.  
	%VZ: modified below
	Then the graph $G-C$ has at least $2b-2$ vertices and, by Lemma \ref{lem:P4C4coclaw}, it is bipartite.
	%The graph $G-C$ is bipartite and has at least $2b-2$ vertices. 
	If this graph has no independent set of size $b$, 
	then in any bipartition of this graph each part contains exactly $b-1$ vertices, and each vertex has a neighbour in the opposite part. But then $|C|=a-2$ and therefore $C$ together with 
	any two adjacent vertices in $G-C$ create a clique of size $a$.   
	 \end{proof}

%%%%%%%%%%%%%%%%%%%%%%%%%%%%%%%%%%%%%%%%%%%%%%%%%%%%%%%%%%%%%%%%%%%%%%%%%%%%%%%%%%%%%%%%%%%%%%%%%%%%%%%%%
\subsection{The class of (co-diamond,paw,claw)-free graphs}
%%%%%%%%%%%%%%%%%%%%%%%%%%%%%%%%%%%%%%%%%%%%%%%%%%%%%%%%%%%%%%%%%%%%%%%%%%%%%%%%%%%%%%%%%%%%%%%%%%%%%%%%%

\begin{lemma}\label{lem:co-diamond}
	Let $G$ be a $($co-diamond,paw,claw$)$-free graph. 
	\begin{itemize}
		\item If $G$ is connected, then it is either a path with at most 5 vertices or a cycle with at most 6 vertices or the complement of a graph of vertex degree at most 1.  
		\item If $G$ has two connected components, then either both components are complete graphs or one of the components is a single vertex and the other is the complement of a graph of vertex degree at most 1.
		\item If $G$ has at least 3 connected components, then $G$ is edgeless. 
	\end{itemize}
\end{lemma}

\begin{proof}
	Assume first that $G$ is connected. It is known (see, e.g., \cite{paw}) that every connected paw-free graphs is either $K_3$-free or complete multipartite. i.e., $\overline{P}_3$-free.
	If $G$ is $K_3$-free, then together with the claw-freeness of $G$ this implies that $G$ has no vertices of degree more than 2, i.e., $G$ is either a path or a cycle. To avoid an induced co-diamond, 
	a path cannot have  more than 5 vertices  and a cycle cannot have  more than 6 vertices. If $G$ is complete multipartite, then each part has size at most 2, since otherwise an induced claw arises. 
	In other words, the complement of $G$ is a graph of vertex degree at most 1. 
	
	Assume now that $G$ has two connected components. If each of them contains an edge, then both components are cliques, since otherwise two non-adjacent vertices in one of the components with
	two adjacent vertices in the other component create an induced co-diamond. If one of the components is a single vertex, then the other is  $\overline{P}_3$-free (to avoid an induced co-diamond) and 
	hence is the complement of a graph of vertex degree at most 1 (according to the previous paragraph).
	
	Finally, let $G$ have at least 3 connected components. If one of them contains an edge, then this edge together with two vertices from two other components form an induced co-diamond. 
	Therefore, every component of $G$ consists of a single vertex, i.e., $G$ is edgeless. 
	 \end{proof}

\begin{theorem}
	For the class $A$ of $($co-diamond,paw,claw$)$-free graphs  and for all $a,b\ge 3$, 
	\begin{itemize}
		\item[] $R_A(a,3)=2a-1,$
		\item[] $R_A(a,b)=\max \{ 2a,b \}$ for $b\ge 4$,
	\end{itemize}
	except for the following four numbers $R_A(3,3)=6$, $R_A(3,4)=R_A(3,5)=R_A(3,6)=7$.
\end{theorem}

\begin{proof}
	We start with the case $b=3$. Since $C_5$ belongs to $A$, $R_A(3,3)=6$, which covers the first of the four exceptional cases.

	Let $a\ge 4$. The graph $2K_{a-1}$ with $2a-2$ vertices has neither cliques of size $a$ nor independent sets of size $3$, and hence $R_A(a,3)\ge 2a-1$. 
	Conversely, let $G\in A$ be a graph with $2a-1\ge 7$ vertices. If $G$ is connected, then according to Lemma~\ref{lem:co-diamond} $G$ is the complement of a graph of vertex degree at most 1,
	and hence $G$ has a clique of size $a$. If $G$ has two connected components both of which are cliques, then one of them has size at least $a$. 
	If $G$ has two connected components one of which is a single vertex, then either the second component has a couple of non-adjacent vertices, in which case an independent set of size 3 arises, 
	or the second component is a clique of size more than $a$. If $G$ has at least 3 connected components, then it contains an independent set of size more than $3$. 
	Therefore, $R_A(a,3)\le 2a-1$ and hence $R_A(a,3)=2a-1$  for $a\ge 4$.
	
%	\medskip
%	\noindent
	From now on, $b\ge 4$.  Consider the last three exceptional cases, i.e., let $a = 3$ and $4 \leq b \leq 6$.
%	\begin{itemize}
%		\item[] 
		The graph $C_6$ that belongs to our class has neither a clique of size 3 nor an independent set of size $b\ge 4$ and hence $R_A(a,b)\ge 7$ in these cases.
		Conversely, let $G\in A$ be a graph with at least 7 vertices. If $G$ is connected, then it is the complement of a graph of vertex degree at most 1 and hence contains a clique of size 3. If $G$ has two 
		connected components each of which is a clique, then one of them has size at least 3. If $G$ has two components one of which is a single vertex, then the other component has at least 6 vertices and 
		also contains a clique of size 3. If $G$ has at least 3 connected components, then $G$ has an independent set of size $4\le b\le 6$. Therefore, $R_A(a,b) = 7$ for $a = 3$ and $4 \leq b \leq 6$.
%	\end{itemize}
	
	In the rest of the proof we assume that either $a\ge 4$ or $b\ge 7$. Denote $m=\max \{ 2a,b \}$. If $m=2a$, then 
	the graph $\overline{(a-1)K_2}+K_1$ with $2a-1$ vertices has neither cliques of size $a$ nor independent sets of size $b\ge 7$. If
	$m=b$, then the edgeless graph with $b-1$ vertices has neither cliques of size $a$ nor independent sets of size $b$. Therefore, $R_A(a,b)\ge m$.
	
	Conversely, let $G$ be a graph with at least  $m\ge 7$ vertices.  If $G$ is connected, then it is the complement of a graph of vertex degree at most 1 and hence contains a clique of size $a$.
	If $G$ has two connected components each of which is a clique, then one of them has size at least $a$. If $G$ has two components one of which is a single vertex, then the other component has at least $2a-1$ vertices and 
	also contains a clique of size $a$. If $G$ has at least 3 connected components, then $G$ has an independent set of size $b$. Therefore, $R_A(a,b) = m$.  
	\end{proof}

\section{Bipartite Ramsey numbers}
\label{sec:bip}

Let $G=(A,B,E)$ be a bipartite graph given together with a bipartition $A\cup B$ of its vertex set into two independent 
sets. We call $A$ and $B$ the parts of $G$. The graph $G$ is {\it complete bipartite},
also known as a {\it biclique}, if every vertex of $A$ is adjacent to every vertex of $B$. A biclique with parts of size $n$ and $m$
is denoted by $K_{n,m}$. By $b(G)$ we denote the biclique number of $G$, i.e., the maximum $p$ such that $G$ contains $K_{p,p}$
as an induced subgraph.

Given a bipartite graph $G=(A,B,E)$, we  denote by $\widetilde{G}$ the bipartite complement of $G$, i.e., the bipartite graph on the same vertex set in which two vertices $a\in A$ and $b\in B$ 
are adjacent if and only if they are not adjacent in $G$. We refer to the bipartite complement of a biclique as {\it co-biclique} and denote by 
$a(G)$  the maximum $q$ such that  $\widetilde{G}$ contains $K_{q,q}$ as an induced subgraph.

The notion of bipartite Ramsey numbers is an adaptation of the notion of Ramsey numbers to bipartite graphs and 
it can be defined as follows.

\begin{definition}\label{def:bipRamNum}
	The bipartite Ramsey number $R^b(p,q)$ is the minimum number $n$ such that 
	for every bipartite  graph $G$ with at least $n$ vertices in each of the parts, $G$ contains $K_{p,p}$, or $\widetilde{G}$ contains $K_{q, q}$.
	%every bipartite graph with at least $n$ vertices in each of the parts contains either $K_{p,p}$ or $\widetilde{K}_{q,q}$.
\end{definition}

It is known (see, e.g., \cite{conlon08}) that $R^b(p, p) \geq 2^{p/2}$, and hence bipartite Ramsey numbers are generally non-linear.
However, similarly to the non-bipartite case, they may become linear when restricted to a specific class $X$ of bipartite graphs. 
We denote bipartite Ramsey numbers restricted to a class $X$ by $R^b_X(p,q)$ and say that
$R^b_X(p,q)$ are {\it linear} in $X$ if there is a constant $k$ such that $R^b_X(p, q) \leq k(p + q)$ for all $p, q$.

Similarly to the non-bipartite case, we will say that graphs in a class $X$ of bipartite graphs  have 
\emph{linear bipartite homogeneous subgraphs} if there exists a constant $c = c(X)$ such that
$\max\{ a(G), b(G) \} \geq c \cdot |V(G)|$ for every $G \in X$.
The following proposition can be proved by analogy with Proposition~\ref{LLRN}.

\begin{proposition}\label{cl:RamNumHom}
Let $X$ be a class of bipartite graphs. Then graphs in $X$ have linear bipartite homogeneous subgraphs if and only if bipartite Ramsey numbers are linear in $X$.
\end{proposition}

Some classes of bipartite graphs with linear bipartite homogeneous subgraphs have been revealed recently in \cite{ATW19}, 
where the authors consider bipartite graphs that do not contain a fixed bipartite graph as 
an induced subgraph \emph{respecting the parts}. The subgraph containment relation respecting the parts can be thought of as 
the containment of \emph{colored} bipartite graphs, where a colored bipartite graph is a bipartite graph given with a fixed
bipartition of its vertices into two independent sets of black and white vertices. A colored bipartite graph $H$ is said to be
an induced subgraph of a colored bipartite graph $G$ if there exists an isomorphism between $H$ and an induced subgraph of $G$
that preserves colors.

A number of related results appeared also in \cite{KPT19}, where the authors study zero-one matrices that do not contain a fixed matrix as a submatrix.
Primarily, they are interested in forbidden submatrices $P$ that guarantee the existence of a square 
homogeneous submatrix of linear size in matrices avoiding $P$, where homogeneous means a submatrix with all its entries being equal.
The problems studied in \cite{KPT19} can be interpreted as questions about homogeneous bipartite subgraphs in
colored and (\emph{vertex}-)\emph{ordered} bipartite graphs 
which do not contain a fixed forbidden colored and ordered bipartite subgraph.  
In this case, the notion of graph containment must preserve not only colors but also vertex order.

In the next sections, we extend some of the results obtained in \cite{ATW19} using the language of Ramsey numbers.
%%%%%%%%%%%%%%%%%%%%%%%%%%%%%%%%%%%%%%%%%%%%%%%%%%%%%%%%%%%%%%%%%%%%%%%%%%%%%%%%%
%%%%%%%%%%%%%%%%%%%%%%%%%%%%%%%%%%%%%%%%%%%%%%%%%%%%%%%%%%%%%%%%%%%%%%%%%%%%%%%%%

\subsection{Classes with non-linear bipartite Ramsey numbers}
\label{sec:non-b}

%%%%%%%%%%%%%%%%%%%%%%%%%%%%%%%%%%%%%%%%%%%%%%%%%%%%%%%%%%%%%%%%%%%%%%%%%%%%%%%%%
%%%%%%%%%%%%%%%%%%%%%%%%%%%%%%%%%%%%%%%%%%%%%%%%%%%%%%%%%%%%%%%%%%%%%%%%%%%%%%%%%

According to Lemma~\ref{lem:E}, classes of graphs without short cycles have non-linear Ramsey numbers.
A similar result holds for bipartite graphs, which can be shown via standard probabilistic arguments. For the sake of completeness,
we provide formal proofs below.
%In the present section, we consider classes of bipartite graphs without short cycles and prove an analog of Lemma~\ref{lem:E}. 
We start with a result, which is an adaptation to the bipartite setting of the classical proof by Erd\H{o}s of the existence of high chromatic number graphs without short cycles.

% It is known (see, e.g., \cite{conlon08}) that $R^b(p, p) \geq 2^{p/2}$, and hence bipartite Ramsey numbers are not linear 
% in the class of all bipartite graphs. Below we show that for every $\ell \geq 2$ bipartite Ramsey numbers are also not linear 
% in the class $X_{\ell}$ of the bipartite graphs that contain no cycles of length at most $\ell$. We start with a theorem,
% which is an adaptation to the bipartite setting of the classical proof by Erd{\H}os of the existence of high chromatic number graphs without short cycles.

\begin{lemma}\label{lem:bipWithoutShortCycles}
	Let $k \geq 4$ and $\varepsilon > 0$.
	Then for any sufficiently large $n$, there exists a bipartite graph $G=(A,B,E)$ 
	with $n$ vertices in each of the parts such that $G$ contains no cycles of length at most $k$, and  
	$\widetilde{G}$ contains no $K_{s,s}$ with $s \geq \varepsilon n$.
\end{lemma}
\begin{proof}	
	Let $n$ be a natural number and let $N = 2n$.
	We set $\delta = \frac{1}{2k}$, $p = (2N)^{\delta - 1}$, and consider the random bipartite graph $G(2N,p)$ 
	(i.e., the probability space of bipartite graphs with two parts $A$ and $B$ each of size $N$ such that every 
	pair of vertices $a \in A, b \in B$ is connected by an edge independently with probability $p$).
	
	Let $Y$ be a random variable equal to the number of cycles of length at most $k$ in $G(2N,p)$.
	The number of potential cycles of length $i$ is at most $\frac{1}{2} (i-1)! {2N \choose i} \leq (2N)^i$, and each
	of them is present with probability $p^i$. Hence
	$$
	\mathbb E[Y] \leq \sum_{i=4}^{k} (2N)^i p^i = \sum_{i=4}^{k} (2N)^{\delta i}.
	$$
	Since $(2N)^{\delta i} = o(N)$ for all $i \leq k$, we conclude $\mathbb E[Y] = o(N)$. 
	%If we choose $N$ so large that $\mathbb E[Y] < \frac{N}{2}$, we get by Markov's inequality
	Hence, for every sufficiently large $N$, we have $\mathbb E[Y] < \frac{N}{2}$, and therefore,  by Markov's inequality,
	$$
		P\left[Y \geq N \right] < \frac{1}{2}.
	$$
	
	Now we estimate the maximum size of a co-biclique in $G(2N,p)$, i.e., $a(G(2N,p))$.
	Let us set $s = \lceil \frac{3}{p} \ln N \rceil$. Then again from Markov's inequality, we have
	$$
		P\left[a(G(2N,p)) \geq s \right] \leq {N \choose s} {N \choose s} (1-p)^{s^2} \leq N^{2s} e^{-ps^2}
		= e^{s(2\ln N -ps)},
	$$
	which tends to zero as $N$ goes to infinity. Thus again, for $N$ sufficiently large, we have 
	$$
		P\left[a(G(2N,p)) \geq s \right] < \frac{1}{2}.
	$$
	
	The above conclusions imply that there exists a graph $G=(A,B,E)$ with $Y < N$ and $a(G) < s$.
	Now we want to destroy all of the $Y$ short cycles by removing one vertex from each of them.
	In order to guarantee that the resulting bipartite graph has many vertices in each of the parts we
	destroy half of the cycles by removing vertices from $A$, and the other half by removing vertices from $B$.
	In this way we remove at most $\frac{N}{2}$ vertices from each of $A$ and $B$,
	%$A$ and at most $\frac{N}{2}$ from $B$, 
	and hence we obtain a graph $G' = (A', B', E')$ with at least $\frac{N}{2} = n$ vertices in each of the parts
	such that $G'$ contains neither cycles of length at most $k$, nor the bipartite complement of $K_{s,s}$
	with $s = \lceil \frac{3}{p} \ln N \rceil = \lceil 3 (2N)^{1-\delta} \ln N \rceil = o(N)$.
	By removing some of the vertices from $G'$ we can obtain a bipartite graph with the same properties, but with exactly
	$n$ vertices in each of the parts.
\end{proof}

From this lemma and Proposition~\ref{cl:RamNumHom} we derive the following conclusion. 
 
\begin{cor}\label{cor:shortCycles}
	For every $k \geq 4$, bipartite Ramsey numbers are not linear in the class of bipartite graphs without cycles of length at most $k$.
\end{cor}

\begin{theorem}\label{thm:forest}
	Let $X$ be a class of bipartite graphs defined by a finite set $M$ of bipartite forbidden induced subgraphs.
	If $M$ does not contain a forest or the bipartite complement of a forest, then bipartite Ramsey numbers are not linear in $X$.
\end{theorem} 
\begin{proof}
If $M$ does not contain a forest, then every graph in $M$ contains a cycle. Let $k$ be the size of a largest induced cycle in graphs in $M$, 
which is a finite number, since $M$ is finite. Then $X$ contains all bipartite graphs without cycles of length at most $k$, and hence 
bipartite Ramsey numbers are not linear in $X$ by Corollary~\ref{cor:shortCycles}.

If $M$ does not contain the bipartite complement of a forest, then bipartite Ramsey numbers are not linear in $X$,
since they are linear in $X$ if and only if they are linear in the class of bipartite complements of graphs in $X$. 
\end{proof}

This result is half analogous to Theorem~\ref{thm:onlyif}. Unfortunately, there is no obvious analog for the second half.
In the non-bipartite case, the second half deals with $P_3$-free graphs and their complements. 
Every $P_3$-free graph consists of disjoint union of cliques, and the most natural analog of this class in the bipartite case 
is the class of $P_4$-free bipartite graphs, which are disjoint union of bicliques. However, bipartite Ramsey numbers
are linear in this class, which is not difficult to see. In the absence of any other natural obstacles for linearity
in the bipartite case, we propose the following conjecture.

\begin{conjecture}\label{conj:forestCoForest}
	Let $X$ be a class of bipartite graphs defined by a finite set $M$ of bipartite forbidden induced subgraphs.
	Then bipartite Ramsey numbers in $X$ are linear if and only if $M$ contains a forest and the bipartite complement of a forest.\footnote{As pointed out by the referees, this conjecture has been solved since we first submitted this article \cite{c4}.}
\end{conjecture}

% We note that a stronger conjecture was proposed in \cite{KPT19}.

% \begin{conjecture}[Kor\'{a}ndi, Pach, Tomon]\label{conj:acyclicMatrices}
% 	Let $P$ be an acyclic zero-one matrix.  Then every $n \times n$ matrix that is both $P$-free and $P^c$-free contains
% 	an $cn \times cn$ homogeneous submatrix, for a suitable constant $c > 0$.
% \end{conjecture}

We note that an analogous conjecture, in the context of homogeneous submatrices, was proposed in \cite{KPT19}.  In the next section, we consider some classes of bipartite graphs excluding a forest and the bipartite complement of a forest,
and show that bipartite Ramsey numbers are linear for them. 

%%%%%%%%%%%%%%%%%%%%%%%%%%%%%%%%%%%%%%%%%%%%%%%%%%%%%%%%%%%%%%%%%%%%%%%%%%%%%%%%%%%%%%%%%%%%%%%%%%%%%%%%%

\subsection{Some classes with linear bipartite Ramsey numbers}

%%%%%%%%%%%%%%%%%%%%%%%%%%%%%%%%%%%%%%%%%%%%%%%%%%%%%%%%%%%%%%%%%%%%%%%%%%%%%%%%%%%%%%%%%%%%%%%%%%%%%%%%%

First, we look at some classes defined by a single bipartite forbidden induced subgraph $H$, which is simultaneously a forest and the bipartite complement of a forest. 
The following theorem characterizes all graphs $H$ of this form, where $F_{p,q}$ denotes the graph
represented in Figure~\ref{fig:f-co-f} and $S_{1,2,3}$ is a tree obtained from the claw by subdividing one of its edges ones and another edge twice (also shown in Figure~{\ref{fig:f-co-f}). 
Implicitly, without a proof, this characterization was given in \cite{Allen09}. It also appeared recently in \cite{ATW19}. 
%We provide an alternative proof, which was obtained independently and is more transparent, in our view. 

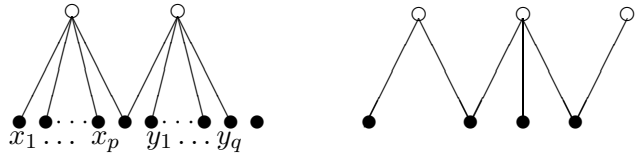
\begin{figure}[ht]
\begin{center} 
\hspace{-1cm}
\begin{picture}(100,50)
\put(0,0){\circle*{5}}
\put(10,0){\circle*{5}}
\put(30,0){\circle*{5}}
\put(40,0){\circle*{5}}
\put(20,42){\circle{5}}
\put(15,0){\circle*{1}}
\put(20,0){\circle*{1}}
\put(25,0){\circle*{1}}
\put(20,40){\line(-1,-2){20}}
\put(20,40){\line(-1,-4){10}}
\put(20,40){\line(1,-2){20}}
\put(20,40){\line(1,-4){10}}
%\put(105,1){(a)}

\put(90,0){\circle*{5}}

\put(50,0){\circle*{5}}
\put(70,0){\circle*{5}}
\put(80,0){\circle*{5}}
\put(60,42){\circle{5}}
\put(55,0){\circle*{1}}
\put(60,0){\circle*{1}}
\put(65,0){\circle*{1}}
\put(60,40){\line(-1,-2){20}}
\put(60,40){\line(-1,-4){10}}
\put(60,40){\line(1,-2){20}}
\put(60,40){\line(1,-4){10}}

\put(-4,-8){$x_1\ldots\ x_p$}
\put(48,-8){$y_1 \ldots y_q$}

\end{picture} 
\hspace{1cm}
\begin{picture}(50,30)
	\setlength{\unitlength}{0.6mm}
	\put(-2,0){\circle*{3}} 
	%\put(10,0){\circle{3}}
	\put(20.5,0){\circle*{3}} 
	\put(32.1,0){\circle*{3}} 
	\put(43.7,0){\circle*{3}}
	\put(55.2,24){\circle{3}} 
	
	\put(32.1,24){\circle{3}}
	
	\put(9,24){\circle{3}}
	\put(9,23){\line(-1,-2){12}} 
	\put(9,23){\line(1,-2){12}}

	\put(32,23){\line(-1,-2){12}} 
	\put(32,23){\line(0,-1){24}} 
	\put(32,23){\line(1,-2){12}} 
	
	\put(55.2,23){\line(-1, -2){12}}
\end{picture}
\end{center}
\caption{The graphs $F_{p,q}$ (left) and $S_{1,2,3}$ (right)}
\label{fig:f-co-f}
\end{figure}

\begin{theorem}\label{acyclic}
A bipartite graph $H$ is a forest and the bipartite complement of a forest if and only if $H$ is an induced subgraph of a $P_7$ or of an $S_{1,2,3}$ or of a graph $F_{p,q}$.
\end{theorem}

The results in \cite{ATW19} and \cite{KPT19} imply that for any natural numbers $p$ and $q$, $F_{p,q}$-free bipartite graphs have linear bipartite homogeneous subgraphs. Hence, by Proposition 2, bipartite Ramsey numbers are linear in the class of $F_{p,q}$-free bipartite graphs. In the next section, we prove that bipartite Ramsey numbers are linear in the class of $S_{1,2,3}$-free bipartite graphs. This leaves an intriguing open question of whether $P_7$-free bipartite graphs have linear Ramsey numbers or not. The structural characterization of the latter graph class from \cite{LZ17} may be helpful in answering this question.

We note that $P_6$ is symmetric with respect to bipartition and it is an induced subgraph of $S_{1,2,3}$. Therefore our result from the next section implies that colored $P_6$-free
bipartite graphs have linear bipartite homogeneous subgraphs, which resolves one of the four open cases from \cite{ATW19}.

\subsubsection{$S_{1,2,3}$-free bipartite graphs}\label{sec:S123}
We start with some definitions.
\begin{itemize}
\item The disjoint union is the operation that creates out of $G_1$ and $G_2$ the bipartite graph $G=(A_1 \cup A_2, B_1 \cup B_2, E_1 \cup E_2)$.
\item The join is the operation that creates out of $G_1$ and $G_2$ the bipartite graph which is the bipartite complement of the disjoint union
of $\widetilde{G}_1$ and $\widetilde{G}_2$
\item The skew join is the operation that creates out of $G_1$ and $G_2$ the bipartite graph $G=G_1 \oslash G_2=(A_1 \cup A_2, B_1 \cup B_2, E_1 \cup E_2 \cup 
\{ab: a \in A_1, b \in B_2\})$. We say $G$ is a skew join of $G_1$ and $G_2$, if either $G=G_1 \oslash G_2$ or $G=G_2 \oslash G_1$.  
\end{itemize}

The three operations define a decomposition scheme, known as canonical decomposition which takes a bipartite graph $G$ and partitions it into
graphs $G_1$ and $G_2$ if $G$ is a disjoint union, join, or skew-join of $G_1$ and $G_2$, and then the scheme applies to $G_1$ and $G_2$, recursively.
Graphs that cannot be decomposed into smaller graphs under this scheme will be called canonically indecomposable.

The following lemma from~\cite{skewstar} characterizes $S_{1,2,3}$-free bipartite graphs containing a $P_7$.
In the paper, the author calls a graph prime if for any two distinct vertices the neighbourhoods are also distinct.

\begin{lemma}\label{containingp7}
A prime canonically indecomposable bipartite $S_{1,2,3}$-free graph $G$ that contains a $P_7$ must be either a path or a cycle or the bipartite complement of either a path or a cycle. 
\end{lemma}

\begin{theorem}
Let $G$ be a canonically indecomposable $S_{1,2,3}$-free bipartite graph that contains a $P_7$. If $G$ has at least $4n$ 
vertices in each part of the bipartition, then $G$ contains a $K_{n,n}$ or a $\widetilde{K}_{n,n}$
\end{theorem}

\begin{proof}
From Lemma~\ref{containingp7} it follows that either $G$ or its bipartite complement
must be either a path or a cycle with some vertices duplicated (as we now no longer assume that $G$ is prime).
Hence, $G=(A,B,E)$ or its bipartite complement must admit a partition $A=A_1 \cup A_2 \cup \ldots \cup A_s$, $B=B_1 \cup B_2 \cup \ldots \cup B_s$
such that: 
\begin{itemize}
\item $A_i, B_i$ are non-empty for all $i \leq s-1$, and at most one of $A_s$ and $B_s$ is empty
\item For any $i \leq s-1$, $A_i$ joined with $B_j$ if $j \in \{i, i+1\}$ and co-joined to $B_j$ otherwise 
\item $A_s$ joined to $B_s$ and $B_1$ and co-joined with $B_j$ for $j \notin \{1,s\}$  
\end{itemize}

Consider first the case when there exists some $i$ such that $|A_i| \geq n$. In this case, if $|B_i \cup B_{i+1}| \geq n$, 
we obtain a biclique $K_{n,n}$ induced by subsets of $A_i$ and $B_i \cup B_{i+1}$. 
On the other hand, if $|B_i \cup B_{i+1}| < n$, then we obtain a $\widetilde{K}_{n,n}$
induced by subsets of $A_i$ and $B \backslash (B_i \cup B_{i+1})$.
Hence, if there exists some $i$ such that $|A_i| \geq n$, then $G$ contains either a $K_{n,n}$ or a $\widetilde{K}_{n,n}$. 
The argument when there exists some $i$ such that $|B_i| \geq n$ is analogous.
 
So assume now that $|A_i|<n$ and $|B_i|<n$ for all $i$. Consider the smallest $k$ such that $|A_1 \cup A_2 \cup \ldots \cup A_k| \geq n$.
If $|B_{k+2} \cup B_{k+3} \cup \ldots \cup B_{s}| \geq n$ then we have a $\widetilde{K}_{n,n}$ induced by subsets of $A_1 \cup A_2 \cup \ldots \cup A_k$
and $B_{k+2} \cup B_{k+3} \cup \ldots \cup B_{s}$. Otherwise, $|B_2 \cup B_3 \cup \ldots \cup B_k| = |B| - |B_1|- |B_{k+1}| - |B_{k+2} \cup \ldots \cup B_s| > 4n-n-n-n=n$ and also $|A_{k+1} \cup A_{k+2} \cup \ldots \cup A_s|=|A| - |A_k| - |A_1 \cup A_2 \cup \ldots \cup A_{k-1}| > 4n - n -n = 2n$. Hence,
we obtain a $\widetilde{K}_{n,n}$ between subsets of $A_{k+1} \cup A_{k+2} \cup \ldots \cup A_{s}$ and $B_2 \cup B_3 \cup \ldots \cup B_k$.  

\end{proof}

\begin{theorem}
Let $X$ be the class of $S_{1,2,3}$-free bipartite graphs. Then $R^b_X(p,q) \leq 6(p+q)$. 
\end{theorem}

\begin{proof}

%Given a set of bipartite graphs $\mathcal{F}$, we denote by $[\mathcal{F}]$ the set of bipartite graphs
%constructed from graphs in $\mathcal{F}$ by means of the following three binary operations. 
%defined for any two vertex-disjoint  bipartite graphs $G_1=(A_1, B_1, E_1)$, $G_2=(A_2, B_2, E_2)$:

%The result from \cite{bipartitedecomposition} states that the class $X$ of $(P_7, S_{1,2,3})$-free bipartite graphs is precisely $[\{K_1\}]$. 
%We will use this fact to show that the class $X$ has linear Ramsey number.

Let $G=(A,B,E)$ be a bipartite graph in $X$ that has $6n$ vertices in each part. If $G_0'=G$ is canonically indecomposable,
then by the previous lemma we can find $K_{n,n}$ or $\widetilde{K}_{n,n}$ in $G$. So assume $G_0'$ is a disjoint a union, a join, or a skew-join of two non-empty graphs $G_1=(A_1, B_1, E_1)$ and $G_1'=(A_1', B_1', E_1')$. Without loss of generality 
we may assume that $|A_1| \leq |A_1'|$.  
Inductively, if $G_k'$ for some $k \in \mathbb{N}$ is not cannonically indecomposable, then $G_k'$ is a disjoint union, a join, or a skew join of two non-empty
graphs $G_{k+1}=(A_{k+1}, B_{k+1}, E_{k+1})$ and $G_{k+1}'=(A'_{k+1}, B'_{k+1}, E'_{k+1})$. 
Again, without loss of generality we may assume that $|A_{k+1}| \leq |A_{k+1}'|$. 

Consider first the case when the procedure stops with canonically indecomposable graph $G_k'$ such that $|A_k'| \geq 4n$. 
If $|B_k'| \geq 4n$, then by the previous lemma,
$G$ contains $K_{n,n}$ or $\widetilde{K}_{n,n}$. On the other hand, if $|B_k'|<4n$,
then we have $|B_1 \cup B_2 \cup \ldots \cup B_k| \geq 2n$ and each vertex in $B_1 \cup B_2 \cup \ldots \cup B_k$ is either joined or co-joined to 
the set $A_k'$. Hence we can find a $K_{n,n}$ or $\widetilde{K}_{n,n}$ induced by subsets of $B_1 \cup B_2 \cup \ldots \cup B_k$ and $A_k'$. 

Now consider the case when the procedure stops with a canonically indecomposable graph $G_k'$ such that $|A_k'| < 4n$. As $A=A_1 \cup A_2 \cup \ldots \cup A_k \cup A_k'$ and $|A|=6n$ it follows that $|A_1 \cup A_2 \cup \ldots \cup A_k|>2n$. Hence, in this case we can pick the smallest $p$ such that $|A_1|+|A_2|+\ldots+|A_p| \geq 2n$.  From the fact that 
$|A_1|+|A_2|+ \ldots + |A_p|+ |A_p'|=6n$ and the definition of $p$ it follows that $|A_p|+|A_p'| \geq 4n$. By construction $|A_p'| \geq |A_p|$, hence $|A_p'| \geq 2n$.  

First, let us consider the case when $|B_p'| \geq n$. Then each vertex of $A_1 \cup A_2 \cup \ldots \cup A_p$ by construction
is either joined of co-joined to $B_p'$. Since $|A_1 \cup A_2 \cup \ldots \cup A_p| \geq 2n$, it is clear that 
we will find either a $K_{n,n}$ or $\widetilde{K}_{n,n}$ in the bipartite graph induced by $A_1 \cup A_2 \cup \ldots \cup A_p$ and $B_p'$.
 
Now, let us consider the case when $|B_p'| < n$. Then each vertex of $B_1 \cup B_2 \cup \ldots B_p$ is either 
joined or co-joined to $A_p'$. Since $|A_p'| \geq 2n$ and $|B_1 \cup B_2 \cup \ldots B_p|>5n$, 
we can find either a $K_{n,n}$ or $\widetilde{K}_{n,n}$ in the bipartite graph induced by $B_1 \cup B_2 \cup \ldots \cup B_p$ and $A_p'$. 

Hence we have shown that $R_X^b(n,n) \leq 6n$ for all $n \in \mathbb{N}$. It now follows easily that for any $p, q \in \mathbb{N}$,
we have $R_X^b(p,q) \leq R_X^b\left(\max\{p,q\}, \max\{p,q\}\right) \leq 6 \cdot \max \{p,q\} \leq 6 (p+q)$.

\end{proof}

%%%%%%%%%%%%%%%%%%%%%%%%%%%%%%%%%%%%%%%%%%%%%%%%%%%%%%%%%%%%%%%%%%%%%%%%%%%%%%%%%%%%%%%%%%%%%%%%%%%%%%%%%
\subsubsection{Exact values of bipartite Ramsey numbers for $P_2 + P_3$-free bipartite graphs}
\label{sec:P2P3}
%%%%%%%%%%%%%%%%%%%%%%%%%%%%%%%%%%%%%%%%%%%%%%%%%%%%%%%%%%%%%%%%%%%%%%%%%%%%%%%%%%%%%%%%%%%%%%%%%%%%%%%%%

Finding exact values of Ramsey numbers is much harder than providing bounds. 
Similarly, finding {\it tight} bounds on the size of homogeneous subgraphs is a very difficult task.  
In \cite{ATW19}, such bounds have been given only for classes where the only  forbidden induced subgraph has two vertices in each part of the bipartition.

In this section, we consider the class of bipartite graphs where the only forbidden induced subgraph $P_2 + P_3$ has two vertices in one of the parts and three in the other.
This class is a subclass of $S_{1,2,3}$-free bipartite graphs, and hence the bipartite Ramsey numbers are linear in this class.
Now we refine this conclusion by deriving {\it exact} values of the bipartite Ramsey numbers for the class of $P_2 + P_3$-free bipartite graphs. For other results concerning this class, see \cite{triangle-free}.

Let $G=(B,W,E)$ be a $P_2+P_3$-free bipartite graph given together with a bipartition of its vertex set into a set $B$ of
black and a set $W$ of white vertices. Also, let $x$ and $y$ be two vertices of $G$ of the same color.

\begin{definition}
	A {\it private neighbour} of $x$ with respect to $y$ is a vertex of $G$ adjacent to $x$ and non-adjacent to $y$. 
	We say $x$ and $y$ are {\it incomparable} if neither $N(x)\subseteq N(y)$ nor $N(y)\subseteq N(x)$, i.e., if both 
	$x$ and $y$ have private neighbours with respect to each other.  
\end{definition}

From the $P_2+P_3$-freeness of $G$ we immediately make the following conclusion.  
\begin{claim}
	If $x$ and $y$ are incomparable, then each of $x$ and $y$ has exactly one private neighbour.
\end{claim}

\begin{definition}
	On the set $B$ of black vertices we define a binary relation $R_B$ such that $(x,y)\in R_B$ if and only if either
	$x$ and $y$ are incomparable or $N(x)=N(y)$. Similarly, we define a binary relation $R_W$ on the set $W$ of white vertices. 
\end{definition}

\begin{lemma}\label{lem:1}
	$R_B$ and $R_W$ are equivalence relations. Moreover, any equivalence class contains either vertices with the same neighbourhood or pairwise incomparable vertices.
\end{lemma}

\begin{proof}
	We prove the lemma for $R_B$. Reflexivity and symmetry are  obvious. To prove transitivity, consider three vertices $x,y,z\in B$
	such that $(x,y)\in R_B$ and $(y,z)\in R_B$. We want to show that $(x,z)\in R_B$. 
	If $N(x)=N(y)$ or $N(y)=N(z)$, transitivity is obvious, and we additionally get that $N(x) = N(y) = N(z)$, since otherwise 
	an induced $P_2 + P_3$ arises. This implies the second part of the statement. 
	
	So assume instead that both pairs 
	$x,y$ and $y,z$ are incomparable. 
	Let $x_y$ be a private neighbour of $x$ with respect to $y$ and 
	let $y_x$ be a private neighbour of $y$ with respect to $x$. To avoid an induced $P_2+P_3$, we conclude that 
	$z$ is adjacent either to both $x_y$ and $y_x$ or to none of them. 
	
	If $z$ is adjacent neither to $x_y$ nor to $y_x$, then $x_y$ is a private neighbour of $x$ with respect to $z$.
	Also, the private neighbour $z_y$ of $z$ with respect to $y$ is different from $x_y$ and is non-adjacent to $x$ (since otherwise 
	$x, y, z, y_x, z_y$ induce a $P_2+P_3$). Therefore, $x$ and $z$ are incomparable, which proves transitivity. We observe that 
	the vertices $x,y,z$ and $x_y,y_x,z_y$ induce a matching and hence $x_y,y_x,z_y$
	also are pairwise incomparable.
	
	%If $z$ is adjacent to both $x_y$ and $y_x$, then $y_x$ is a private 
	%neighbour of $z$ with respect to $x$. Also, the private neighbour $y_z$ of $y$ with respect to $z$ is adjacent to $x$
	%(since otherwise an induced $P_2+P_3$ arises),
	%i.e., $y_z$ is a private neighbour of $x$ with respect to $z$. Therefore, $x$ and $z$ are incomparable, which proves transitivity.
	%We observe that the vertices $x,y,z$ and $x_y,y_x,y_z$ induce a co-matching (the bipartite complement of a matching) and hence $x_y,y_x,y_z$
	%also are pairwise incomparable.  
	
	We can reduce the case where $z$ is adjacent to both $x_y$ and $y_x$ to the previous one by passing to the (bipartite) complement, and noting that $P_2 + P_3$ and $2K_2$ are their own complements. 
	
\end{proof}

\begin{lemma}\label{lem:2}
	For any two equivalence classes $X$ and $Y$ of vertices of the same color,
	\begin{itemize}
		\item either $N(x)\subset N(y)$ for all pairs $x\in X$ and $y\in Y$,
		\item or $N(x)\supset N(y)$ for all pairs $x\in X$ and $y\in Y$.
	\end{itemize}
\end{lemma}

\begin{proof}
	If $x$ and $y$ belong to different equivalence classes $X$ and $Y$, then by definition $N(x)\subset N(y)$ or $N(x)\supset N(y)$.
	Assume without loss of generality that $N(x)\subset N(y)$. Then $N(x)\subset N(y')$ for any vertex $y'\in Y$, since otherwise 
	$N(y')\subset N(x)\subset N(y)$, in which case $y$ and $y'$ are not equivalent. In turn this implies that $N(x')\subset N(y')$  
	for any vertex $x'\in X$. 
\end{proof}

As shown in Lemma~\ref{lem:1}, any equivalence class contains either vertices with the same neighbourhood or pairwise incomparable vertices. We will refer to those as type 1 and type 2 classes respectively. 
Without loss of generality we will assume that any equivalence class of size 1 is of type 1. 
Moreover, we have the following:

\begin{remark} \label{rem:matching}
	For each equivalence class $X_B\subseteq B$ of type 2, 
	there is a corresponding equivalence class $X_W\subseteq W$ of the same type and of the same size such that $X_B\cup X_W$ induces 
	either a matching or a co-matching. Indeed, this follows from the proof of Lemma~\ref{lem:1}: as shown in the last two paragraphs of the proof, any 3 vertices in a class $X_B \subseteq B$ of type 2 have 3 vertices in a class $X_W \subseteq W$ with which they form either a $3K_2$ (case 1) or its bipartite complement (case 2). It is easy to see that $X_W$ only depends on $X_B$, that only one of the two cases can happen between a pair of bags, and that this implies the entire bags themselves induce a matching or a co-matching.
\end{remark}

We call $(X_B,X_W)$ a {\it block} in $G$.  

\begin{lemma}\label{lem:3}
	Let $X$ be an equivalence class and let $v$ be a vertex of the opposite color not belonging to the block containing $X$ (if $X$ is of type 2).
	Then $v$ is either complete or anti-complete to $X$.
\end{lemma} 

\begin{proof}
	If $X$ is of type 1, then $v$ is either complete or anti-complete to $X$ by definition. If $X$ is of type 2, then 
	$v$ cannot distinguish two vertices of $X$, since otherwise a $P_2+P_3$ arises (remember that any two vertices of $X$ belong to an induced $2K_2$).
\end{proof}

Lemma~\ref{lem:2} allows us to order the equivalence classes $B_1,\ldots, B_k$ of $B$ so that $i<j$ implies  $N(x)\supset N(y)$ 
for all pairs $x\in B_i$ and $y\in B_j$. Similarly, we order the equivalence classes $W_1,\ldots, W_p$ of $W$ so that $i<j$ implies that  $N(x)\subset N(y)$ 
for all pairs $x\in W_i$ and $y\in W_j$. Then we order the vertices within equivalence classes arbitrarily. In this way,
we produce a linear order of $B$ and a linear order of $W$ satisfying the following properties: 
\begin{itemize}
	\item if a vertex $x\in B$ in an equivalence class of type 1 is adjacent to a vertex $y\in W$, 
	then $x$ is adjacent to all the white vertices following $y$ in the order;
	\item if a vertex $y\in W$ in an equivalence class of type 1 is adjacent to a vertex $x\in B$, 
	then $y$ is adjacent to all the black vertices preceding $x$ in the order;
	\item if $x$ is a black vertex in a block $(X_B,X_W)$, then $x$ is adjacent to all the white vertices that appear after $X_W$
	and non-adjacent to all the white vertices that appear before $X_W$;
	\item if $y$ is a white vertex in a block $(X_B,X_W)$, then $y$ is adjacent to all the black vertices that appear before $X_B$ and 
	non-adjacent to all the black vertices that appear after $X_B$.
\end{itemize}

Also, without loss of generality, we assume that in a block $(X_B,X_W)$ the order of $X_B$ is consistent with the order of $X_W$ 
(according to the bijection defined by the matching or co-matching between  $X_B$ and $X_W$).

\begin{theorem}
	For the class $A$ of $P_2+P_3$-free bipartite graphs and for all $p,q \geq 2$, $R^b_A(p,q)=\max \{p,q\}+p+q-2$.
\end{theorem}

\begin{proof}
	To prove that $R^b_A(p,q)\ge\max \{p,q\}+p+q-2$, assume, without loss of generality, that $q=\max\{p,q\}$ (if $p=\max\{p,q\}$, the proof is similar). 
	Let $G=(B,W,E)$ be a $P_2+P_3$-free bipartite graph with $|B|=|W|=2q+p-3$ 
	such that $B=B_0\cup B_1$, $W=W_1\cup W_2$, where $|B_0|=|W_2|=p-2$, $B_1\cup W_1$ is an induced matching of size $2q-1$,
	$B_0$ is complete to $W$, while $W_2$ is complete to $B$.
	
	Assume $G$ contains a biclique $K_{p,p}$. Then this biclique contains at least 2 vertices in $B_1$ and at least two vertices 
	in $W_1$. But then $B_1\cup W_1$ is not an induced matching. This contradiction shows that $G$ is $K_{p,p}$-free.
	
	Assume $G$ contains a co-biclique $\widetilde{K}_{q,q}$. This co-biclique cannot contain vertices of $B_0$ or $W_2$ (since these vertices
	dominate the opposite part of the graph). But then we obtain a contradiction to the assumption that the size of the 
	matching $B_1\cup W_1$ is $2q-1$. Therefore, $G$ is $\widetilde{K}_{q,q}$-free. This proves the inequality 
	$R^b_A(p,q)\ge\max\{p,q\}+p+q-2$.
	
	To prove the reverse inequality, consider an arbitrary $P_2+P_3$-free bipartite graph $G=(B,W,E)$ with $|B|=|W|=\max\{p,q\}+p+q-2$.
	Without loss of generality we assume that the vertices of $G$ are ordered as described above. 
	Denote by 
	\begin{itemize}
		\item $B_1$ the set of the first $p$ vertices of $B$,
		\item $B_3$ the set of the last $q$ vertices of $B$,
		\item $B_2=B-(B_1\cup B_3)$,
		\item $W_1$ the set of the first $q$ vertices of $W$,
		\item $W_3$ the set of the last $p$ vertices of $W$,
		\item $W_2=W-(W_1\cup W_3)$.
	\end{itemize}
	
%	We assume that $\max\{p,q\}\ge 2$ (since the case $\max\{p,q\}=1$ is trivial), and hence $|B_2|=|W_2|=\max\{p,q\}-2\ge 0$, i.e., 
%	$B_1$ and $B_3$ are disjoint and $W_1$ and $W_3$ are disjoint.
	Since $p,q \geq  2$, we have that $|B_2|=|W_2|=\max\{p,q\}-2\ge 0$, i.e.,	$B_1$ and $B_3$ are disjoint and $W_1$ and $W_3$ 
	are disjoint.
	
	If $B_1$ is complete to $W_3$ or $W_1$ is anti-complete to $B_3$, then $G$ contains $K_{p,p}$ or $\widetilde{K}_{q,q}$, respectively. 
	Therefore, we assume there is a pair $w_1\in W_1$, $b_3\in B_3$ of adjacent vertices and a pair $b_1\in B_1$, $w_3\in W_3$ of non-adjacent vertices.
	
	If the vertices $w_1$ and $b_3$ do {\it not} belong to the same block, then every black vertex that appears before $b_3$ 
	is adjacent to every white vertex that appears after $w_1$, in which case an induced $K_{p,p}$ arises. Therefore, we assume 
	that $w_1$ and $b_3$  belong to the same block. Similarly, we assume that $b_1$ and $w_3$ belong to the same block. 
	It is not difficult to see that  all four vertices belong to the same block. We denote this block by $T$.
	
	In what follows we assume that $T$ is an induced matching (the case when $T$ is a co-matching is symmetric).
	If $|T|\ge 2q$ (i.e., if $T$ contains at least $2q$ edges), then $G$ contains an induced $\widetilde{K}_{q,q}$. 
	Therefore, from now on we assume that $|T|\le 2q-1$.
	
	%VZ: modified below
	Since each of  $B_1$ and $B_3$ contains a vertex of $T$, we conclude that all vertices of $B_2$ belong to $T$. Hence the number of 
	black vertices of $T$ that appear before $B_3$ is at least $\max\{p,q\}-1$. Similarly, the number of white vertices of $T$ that appear after 
	$W_1$ is at least $\max\{p,q\}-1$.
%	By assumption, $W_1$ contains a vertex of $T$. Therefore, there exist at least $p$ white vertices (the vertices of $W_3$) that either belong to $T$
%	or appear after $T$. As a result, if there exist at least $p$ black vertices that appear before $T$, then $G$ contains an induced $K_{p,p}$.
%	This allows us to assume that the number of black vertices that appear before $T$ is at most $p-1$ and hence the number of black vertices
%	of $T$ that appear before $B_3$ is at least $\max\{p,q\}-1$. Similarly, the number of white vertices of $T$ that appear after $W_1$
%	is at least $\max\{p,q\}-1$. 
	Therefore, $|T|\ge 2\max\{p,q\}-1$ (at least $\max\{p,q\}-1$ edges before $w_1b_3$, 
	at least $\max\{p,q\}-1$ edges after $w_1b_3$ plus the edge $w_1b_3$). 
	Combining $|T|\ge 2\max\{p,q\}-1\ge 2q-1$ and $|T|\le 2q-1$, 
	we conclude that $|T|=2q-1$. 
	
	%VZ: modified below
	If there exist at least one black vertex that appears after $T$, then this vertex together with the last $q-1$ black vertices in $T$ and the first
	$q$ white vertices in $T$ would induce a $\widetilde{K}_{q,q}$. Similarly, if there is at least one white vertex before $T$, then 
	an induced $\widetilde{K}_{q,q}$ can be easily found.
%	If there exist at least one black vertex that appears after $T$ or at least one white vertex that appears before $T$, then 
%	an induced $\widetilde{K}_{q,q}$ can be easily found.  
	Therefore, we assume that there exist 
	$$
	\max\{p,q\}+p+q-2-(2q-1)=\max\{p,q\}+p-q-1\ge p-1
	$$
	black vertices before $T$ and at least $p-1$ white vertices after $T$. These vertices together with the edge $w_1b_3$
	create an induced $K_{p,p}$.   
\end{proof}


\begin{thebibliography}{99}
	%%%%%%%%%%%%%%%%%%%%%%
	
	\bibitem{Allen09}
	P. Allen, Forbidden induced bipartite graphs.
	{\it J. Graph Theory}, 60(3) (2009), 219--241.
	
	
	\bibitem{method}
	N. Alon, J. H. Spencer, The probabilistic method.  John Wiley \& Sons, 2004.
	
	% \bibitem{bipartite}
	% Beineke, L.W., Schwenk, A.J., On a bipartite form of the Ramsey problem, Proc. 5th
	% British Combin. Conf. 1975, Congressus Numer. XV, 1976, pp. 17--22.

	%\bibitem{Aistis}
	%A. Atminas, Classes of graphs without star forests and related graphs. (2017),
	%arXiv:1711.01483. 
	
	\bibitem{IWOCA2018}
	A. Atminas, V. Lozin, V. Zamaraev, 
	Linear Ramsey numbers.
	{\it Lecture Notes in Computer Science}, 10979 (2018) 26--38.
	
         \bibitem{ATW19}
	M. Axenovich, C. Tompkins, L. Weber,
	Large homogeneous subgraphs in bipartite graphs with forbidden induced subgraphs. (2019),
	arXiv:1903.09725.

	\bibitem{first}
	R. Belmonte, P. Heggernes,  P. van 't Hof, A. Rafiey, R. Saei,
	Graph classes and Ramsey numbers. 
	{\it Discrete Appl. Math}. 173 (2014), 16--27.
	
	\bibitem{pseudo-split}
	Z. Bl\'azsik, M.  Hujter, A. Pluh\'ar, Z. Tuza, 
	Graphs with no induced $C_4$ and $2K_2$.
	{\it Discrete Math.}, 115 (1993) 51--55.
	
	\bibitem{cs}
	M. Chudnovsky, P. Seymour, 
	Extending the Gy\'arf\'as-Sumner conjecture.
	{\it J. Combin. Theory,  B}, 105 (2014) 11--16.

	\bibitem{triangle-free} M. Chudnovsky, P. Seymour, S. Spirkl, M. Zhong, Triangle-free graphs with no six-vertex induced path. {\it Discrete Math.}, 341(8) (2018), 2179--2196.

	\bibitem{conlon08}
	D. Conlon,
	A new upper bound for the bipartite Ramsey problem.
	{\it J. Graph Theory}, 58(4) (2008) 351--356.
		
	\bibitem{cographs}
	D.G. Corneil, H. Lerchs, B.L. Stewart,
	Complement reducible graphs.
	{\it Discrete Appl. Math}, 3 (1981) 163--174.
	
%	\bibitem{chromcochrom} 
%	P.  Erd\H{o}s, J. Gimbel, H.J. Straight, 
%	Chromatic number versus cochromatic number in graphs with bounded clique number.
%	{\it European J. Combinatorics}, 11 (1990) 235--240.
%	
	\bibitem{EH-conjecture} 
	P.  Erd\H{o}s, A. Hajnal, 
	Ramsey-type theorems.
	{\it Discrete Appl. Math}. 25 (1989) 37--52.
	
	
	\bibitem{split}
	S. Foldes, P.L. Hammer,  Split graphs.  
	{\it  Congressus Numerantium}, No. XIX, (1977) 311--315.

	%\bibitem{bipartitedecomposition}
	%J-L. Fouquet, V. Giakoumakis, J.M. Vanherpe,
	%Bipartite graphs totally decomposable by canonical decomposition,
	%{\it Inter. J. Foundations of Computer Science}, 10 (1999), 513-533
	
	\bibitem{GRR00}
	R.L. Graham, V. R{\"o}dl, A. Ruci{\'n}ski,
	On graphs with linear Ramsey numbers.
	{\it J. Graph Theory}, 35(3) (2000) 176--192.
	
	\bibitem{chi}
	A. Gy\'arf\'as, On Ramsey covering-numbers. {\it Infinite and Finite Sets} 2 (1975) 801--816.
	
	%	A. Gy\'arf\'as, A. On Ramsey covering-numbers. Infinite and finite sets (Colloq., Keszthely, 1973; dedicated to P. Erd\H os on his 60th birthday), 
	% Vol. II, pp. 801--816. Colloq. Math. Soc. Jan\'os Bolyai, Vol. 10, North-Holland, Amsterdam, 1975. 
	
    \bibitem{KPT19}
	D.~Kor\'{a}ndi, J.~Pach, I.~Tomon,~Large homogeneous submatrices. (2019),
	arXiv:1903.06608v2.

	
	\bibitem{lovasz1978}
	L.~Lov{\'a}sz, Kneser's conjecture, chromatic number, and homotopy. 
	{\it J. Combin. Theory, A}. 25(3) (1978) 319--324.	
	
	% \bibitem{kp}
	% Kierstead, H.A.; Penrice, S.G. 
	% Radius two trees specify $\chi$-bounded classes. 
	% {\it Journal of Graph Theory} 18 (1994), 119--129. 

	\bibitem{skewstar}
	V. Lozin,
	Bipartite graphs without a skew star.
	{\it Discrete Math.}, 257 (2002) 83-100.
	
	\bibitem{LZ17} 
	V. Lozin, V. Zamaraev,
	The structure and the number of $P_7$-free bipartite graphs.
	{\it European J. Combinatorics}, 65 (2017) 143--153.

%	\bibitem{efree}
%	V. Lozin,
%	E-free bipartite graphs,
%	{\it Diskretnyj Analiz i Issledovanie Operatsij.}, 1 (2000)	


	\bibitem{paw}
	S. Olariu,  Paw-free graphs. 
	{\it Information Processing Letters}, 28 (1988) 53--54.
	
	
	\bibitem{Ramsey} 
	F.P. Ramsey, On a problem of formal logic. 
	{\it Proceedings of the London Mathematical Society}, 30 (1930), 264--286.
	
	% \bibitem{Scott}
	% A.D. Scott, Induced trees in graphs of large chromatic number, J. Graph Theory 24 (1997), 297--311.
	
	
	\bibitem{frac} E. Scheinerman,  D. Ullman, Fractional graph theory. 
	A rational approach to the theory of graphs. Dover Publications, Inc., Mineola, NY, 2011. xviii+211 pp.	 
	
	\bibitem{c4} A. Scott, P. Seymour, S. Spirkl, Pure pairs. IV. Trees in bipartite graphs. (2020), arXiv:2009.09426.
	
	\bibitem{planar}
	R. Steinberg, C.A.  Tovey,
	Planar Ramsey numbers. 
	{\it J. Combin. Theory, B} 59 (1993) 288--296.
	
	\bibitem{sumner}
	D. P. Sumner, Subtrees of a graph and chromatic number. In: G. Chartrand (Ed.), The theory and applications of graphs. John Wiley \& Sons, New York, 1981, 557--576.

\end{thebibliography}
\end{document}